\documentclass[12pt,reqno]{amsart}
\usepackage{}

\usepackage{amsmath}
\allowdisplaybreaks
\usepackage{amsfonts}
\usepackage{amssymb}
\usepackage[all]{xy}           

\usepackage{bbding}
\usepackage{txfonts}
\usepackage{amscd}

\usepackage[shortlabels]{enumitem}
\usepackage{ifpdf}
\ifpdf
  \usepackage[colorlinks,final,backref=page,hyperindex]{hyperref}
\else
  \usepackage[colorlinks,final,backref=page,hyperindex,hypertex]{hyperref}
\fi
\usepackage{tikz}
\usepackage[active]{srcltx}

\makeatletter

\newtheorem{df}{Definition}[section]
\newtheorem{thm}{Theorem}[section]
\newtheorem{cor}{Corollary}[section]
\newtheorem{rem}{Remark}[section]

\newtheorem{prop}{Proposition}[section]
\newtheorem{exa}{Example}[section]
\newtheorem{lem}{Lemma}[section]

\numberwithin{equation}{section}

\begin{document}

\date{}
\title
{On noncommutative Hom-anti-pre-Poisson superalgebras and related structures}

\author{W. Ben Abdelhafidh and O. Ncib }

\address{University of Sfax, Faculty of Sciences of Sfax,   BP 1171, 3000 Sfax, Tunisia}

\email{wiembenabdelhafidh@gmail.com}

\address{University of Gafsa, Faculty of Sciences Gafsa, 2112 Gafsa, Tunisia}

\email{othmenncib@yahoo.fr, othmen.ncib@fsgf.u-gafsa.tn}



\begin{abstract}

In this paper, we develop the theory of several Hom-type anti-algebraic structures in the $\mathbb{Z}_2$-graded setting. We introduce and study Hom-anti-associative superalgebras, Hom-anti-dendriform superalgebras, and Hom-anti-pre-Lie superalgebras, which arise as Hom-type generalizations of the corresponding anti-algebraic structures. Various constructions and examples are provided, together with structural properties and representation-theoretic aspects. In particular, we investigate the role of anti-super-$\mathcal{O}$-operators and anti-Rota-Baxter operators in the construction of Hom-anti-dendriform superalgebras. Furthermore, we introduce the notion of noncommutative Hom-pre-Poisson superalgebras and noncommutative Hom-anti-pre-Poisson superalgebras as Hom-type extensions of noncommutative Poisson-type structures in the graded framework. These structures naturally combine Hom-anti-pre-Lie and Hom-anti-dendriform superalgebras through suitable compatibility conditions. Several relationships between the introduced structures are established, providing a unified framework for studying twisted and graded generalizations of noncommutative Poisson-type algebras.
\end{abstract}





\renewcommand{\thefootnote}{\fnsymbol{footnote}}
\footnote[0]{ Corresponding author (Othmen Ncib): othmenncib@yahoo.fr, othmen.ncib@fsgf.u-gafsa.tn}

 \maketitle{}
\tableofcontents

\section{Introduction}


The notion of Hom-associative algebra, introduced in \cite{Makhlouf-Silvestrov1}, is a generalization of associative algebras, it corresponds to the notion of Hom-Lie algebras, which appeared in physics, in the sense that the commutator
of a Hom-associative algebra gives a Hom-Lie algebra.
It was originally introduced by Makhlouf and Silvestrov in \cite{Makhlouf-Silvestrov2}. Hom-Lie algebras and more general quasi-Hom-Lie algebras were introduced first
by Hartwig, Larsson and Silvestrov in \cite{Hartwig-Larsson-Silvestrov} where a general approach to discretization
of Lie algebras of vector fields using general twisted derivations ($\sigma$-derivations) and a
general method for construction of deformations of Witt and Virasoro type algebras
based on twisted derivations have been developed. The Hom-alternative
algebras were introduced by A. Makhlouf in \cite{Makhlouf}. Several applications for this notion are given in \cite{Hamdedi-Makhlouf1,Hamdedi-Makhlouf2}. In these works, the authors established the foundations of the formal deformation theory and the connection with cohomology for (Hom)-alternative algebras and  used a composition construction to produce deformations of left alternative algebras into left Hom-alternative algebras. The results were also  extended  to Hom-Malcev algebras. In \cite{Yau}, the authors introduced the Hom-Jordan algebras as twisted generalizations of Jordan algebras.\\

Pre-Lie algebras (called also left-symmetric algebras) were first mentioned by Cayley in $1890$ \cite{Cayley} as a kind of rooted tree algebra and later
arose again from the study of convex homogeneous cones \cite{E.B. Vinberg}, affine manifold and affine structures on
Lie groups \cite{Koszul}, and deformation of associative algebras \cite{Gerstenhaber}. They also appeared in many
fields in mathematics and mathematical physics, such as complex and symplectic structures on Lie groups and Lie algebras \cite{And-Sal, B.Y.Chu,Dardié-Medina1,Dardié-Medina2,Lichnerowicz-Medina}, phases spaces of Lie algebras \cite{Bai,B.A.Kupershmidt},
integrable systems \cite{M.Bordemann}, classical and quantum Yang-Baxter equations \cite{Diatta-Medina}, combinatorics
\cite{Ebrahimi-Fard}, Poisson brackets and infinite dimensional Lie algebras, vertex algebras, quantum field
theory \cite{Connes-Kreimer}, and operads \cite{Chapoton-Livernet}. Hom-pre-Lie algebras are a twisted analogue of pre-Lie algebras, where the pre-Lie algebra identity is twisted by a self linear map, called the structure map.
This notion was introduced in \cite{Makhlouf-Silvestrov1}.  Recently, Hom-pre-Lie algebras have been studied from
several aspects: The geometrization of Hom-pre-Lie algebras was studied in \cite{Zhang-Yu-Wang}, universal $\alpha$-central extensions of Hom-pre-Lie algebras were studied in \cite{Sun-Chen-Zhou} and the bialgebra theory of Hom-pre-Lie algebras
was studied in \cite{Sun-Li}.

The super-case (or $\mathbb{Z}_2$-graded analogue) of pre-Lie algebras, known as pre-Lie superalgebras, first emerged in the study of cohomological structures of associative algebras \cite{Chapoton-Livernet}. These algebraic structures exhibit deep and significant connections with various areas of mathematics and mathematical physics. In particular, they play an important role in super-symplectic geometry, vertex superalgebras, and graded versions of the classical Yang-Baxter equation. For further developments and detailed discussions, we refer the reader to \cite{Kong-Bai,
Kong-Chen-Bai}. Hom-pre-Lie algebras are closely related to Hom-Lie algebras through a natural construction. More precisely, every Hom-pre-Lie algebra $(\mathcal{A},\circ,\alpha)$ induces a Hom-Lie algebra $(\mathcal{A},[\cdot,\cdot]^C,\alpha)$, where the bracket is defined by the commutator
\[
[x,y]^C = x \circ y - y \circ x,\quad \forall x,y \in \mathcal{A}.
\]
This Hom-Lie algebra is referred to as the subadjacent Hom-Lie algebra associated with $(\mathcal{A},\circ,\alpha)$ and is commonly denoted by $\mathcal{A}^C$. Hom-pre-Lie algebras play a fundamental role in the structural theory of Hom-Lie algebras, particularly in the study of their representations, deformations, and related algebraic constructions. Hom-pre-Lie superalgebras are a natural $\mathbb{Z}_2$-graded generalization of Hom-pre-Lie algebras obtained by equipping a superspace with a bilinear product and a twisting linear map. They can be viewed as Hom-type deformations of pre-Lie superalgebras, in which the defining identity is twisted by a linear self-map. These structures provide an appropriate framework for studying Hom-type generalizations of Lie superalgebras, since the super-commutator associated with a Hom-pre-Lie superalgebra naturally induces a Hom-Lie superalgebra, called the subadjacent Hom-Lie superalgebra. Moreover, Hom-pre-Lie superalgebras play an important role in the representation theory, cohomology, and deformation theory of Hom-Lie superalgebras. They also appear naturally in the construction of Hom-Lie superalgebras via $\mathcal{O}$-operators and Rota-Baxter operators (see, for example, \cite{O-N1}).\\

Anti-pre-Lie algebras, which constitute an anti-version of pre-Lie algebras, were first introduced by G. Liu and C. Bai \cite{Liu-Bai1}. Subsequently, the superalgebraic extension of these structures was developed by Zhao Chen, Shanshan Liu, and Liangyun Chen in \cite{Zhao-Liu-Chen}. An anti-pre-Lie superalgebra is a $\mathbb{Z}_2$-graded vector space $\mathcal{A}=\mathcal{A}_{\bar{0}} \oplus \mathcal{A}_{\bar{1}}$ equipped with an even bilinear product 
$\circ : \mathcal{A} \times \mathcal{A} \to \mathcal{A}$ satisfying, for all homogeneous elements $x,y,z \in \mathcal{A}$,
\begin{equation}\label{def-anti-pre-Lie superalgebras1}
x\circ(y\circ z)-(-1)^{|x||y|}y\circ(x\circ z)=(-1)^{|x||y|}[y,x]\circ z,
\end{equation}
\begin{equation}\label{def-anti-pre-Lie superalgebras2}
[x,y]_\circ\circ z+(-1)^{|x|(|y|+|z|)}[y,z]_\circ\circ x+(-1)^{|z|(|x|+|y|)}[z,x]_\circ\circ y=0,
\end{equation}
where
\begin{equation}\label{def-anti-pre-Lie superalgebras3}
[x,y]_\circ=x\circ y-(-1)^{|x||y|}y\circ x.
\end{equation}
 This identity is called the anti-pre-Lie super-identity. Moreover, the super-commutator $[x,y]_\circ$ 
defines a Lie superalgebra structure on $\mathcal{A}$. In this paper we investigate the Hom-type generalization of this structure which is Hom-anti-pre-Lie superalgebra. \\

The concept of dendriform algebras was introduced by Loday in the context of algebraic $K$-theory \cite{Loday1}. Since their introduction, dendriform algebras have played a fundamental role in various areas of mathematics and mathematical physics, including arithmetic \cite{Loday2}, combinatorics \cite{Loday3}, Hopf algebras \cite{Chap, Holtkamp1, Holtkamp2, Loday4, Ronco}, homological algebra \cite{Frabetti1, Frabetti2}, operad theory \cite{Loday5}, Lie and Leibniz algebras \cite{Frabetti2}, and quantum field theory \cite{Foissy}. One of the essential structural properties of a dendriform algebra $(\mathcal A,\succ,\prec)$ is that the sum of its two binary operations,
\[
x\cdot y := x\succ y + x\prec y,\quad \forall x,y\in \mathcal A,
\]
defines an associative algebra structure on $A$. This property is commonly interpreted as a splitting of the associativity condition into two finer compatibility relations.

Furthermore, dendriform algebras are closely related to pre-Lie algebras, which form an important class of Lie-admissible algebras whose commutators naturally give rise to Lie algebras. These structures appear in many areas of mathematics and physics (see \cite{Bai1, Burde} and references therein). More precisely, for any dendriform algebra $(\mathcal A,\prec,\succ)$, the bilinear operation defined by
\begin{equation}\label{eq:pre}
x \circ y = x\prec y - y\succ x,\quad \forall x,y\in \mathcal A,
\end{equation}
endows $\mathcal A$ with a pre-Lie algebra structure $(\mathcal A,\circ)$, called the associated pre-Lie algebra of $(\mathcal A,\succ,\prec)$. Consequently, dendriform algebras provide a natural framework connecting associative algebras, pre-Lie algebras, and Lie algebras. This relationship can be summarized by the following commutative diagram of categories \cite{chap1}:
\begin{equation}\label{eq:dendri}
\begin{matrix}
{\rm dendriform\quad algebras} & \longrightarrow & {\rm pre\text{-}Lie\quad algebras} \\
\downarrow & & \downarrow \\
{\rm associative\quad algebras} & \longrightarrow & {\rm Lie\quad algebras}.
\end{matrix}
\end{equation}

Motivated by the study of new splittings of associative products, Dongfang Gao, Guilai Liu and Chengming Bai introduced the notion of anti-dendriform algebras (see \cite{Dongfang-Liu-Bai} for more details). An anti-dendriform algebra is a vector space endowed with two binary operations whose sum defines an associative product and whose multiplication operators satisfy specific compatibility conditions reflecting an 'anti' version of the classical dendriform relations. This
construction provides a new approach to decomposing algebraic operations.
Moreover, anti-dendriform algebras are closely related to operator forms such as anti-$\mathcal O$-operators and
anti-Rota-Baxter operators. They also exhibit strong connections with Novikov-type structures, which leads to new insights into the theory of Novikov algebras and their generalizations. These
developments highlight the importance of anti-dendriform algebras in the study of new algebraic
splittings and related algebraic structures. 
The aim of this paper is to introduce the notion of Hom-anti-dendriform superalgebras, which can be viewed as a Hom-type graded generalization of anti-dendriform algebras. We study their basic properties and investigate their relationships with other algebraic structures, providing new examples and constructions arising from suitable operator forms. The classical relationships between dendriform, pre-Lie, associative, and Lie algebras, summarized in diagram \eqref{eq:dendri}, can be extended to the anti-dendriform setting in the Hom-super framework. Indeed, by introducing the Hom-anti-dendriform and Hom-anti-pre-Lie superalgebras, one obtains a natural generalization where the ``anti'' versions replace the classical structures, and the Hom-twist accommodates graded and deformation effects.

The extended diagram can be displayed as follows:

\begin{equation}\label{eq:antiHomSuper}
\begin{matrix}
{\rm Hom\text{-}anti\text{-}dendriform\ superalgebras} & \longrightarrow & {\rm Hom\text{-}anti\text{-}pre\text{-}Lie\ superalgebras} \\
\downarrow & & \downarrow \\
{\rm Hom\text{-}associative\ superalgebras} & \longrightarrow & {\rm Hom\text{-}Lie\ superalgebras}.
\end{matrix}
\end{equation}

Here, the horizontal arrows correspond to constructions associating anti-pre-Lie superalgebras to anti-dendriform superalgebras, while the vertical arrows reflect the summation of the defining operations, generalizing the classical splittings to the Hom-super setting.\\

Noncommutative Poisson algebras were introduced by Xu \cite{Xu} as a natural noncommutative generalization of classical Poisson algebras, where the associative product is not required to be commutative but remains compatible with a Lie bracket through the Leibniz rule. Their $\mathbb{Z}_2$-graded analogue, called noncommutative Poisson superalgebras, extends this structure to the super setting and appears naturally in the study of graded noncommutative geometry and mathematical physics.

As a refinement of Poisson-type structures, pre-Poisson algebras were introduced by Aguiar \cite{Aguiar}. In the graded context, noncommutative pre-Poisson superalgebras provide a splitting of the identities of noncommutative Poisson superalgebras through compatible pre-Lie type operations. 

Motivated by the theory of Hom-algebras initiated by Hartwig, Larsson and Silvestrov \cite{Hartwig-Larsson-Silvestrov}, Hom-type analogues of these structures have been investigated by twisting the defining identities via a linear map. In particular, Hom-noncommutative Poisson superalgebras and Hom-noncommutative pre-Poisson superalgebras provide a natural framework for studying deformations and twisted symmetries of graded noncommutative Poisson structures.

In this paper, we introduce the notion of noncommutative Hom-anti-pre-Poisson superalgebras, which can be viewed as a natural generalization of both Hom-anti-pre-Lie superalgebras and Hom-anti-dendriform superalgebras. These structures provide a unified algebraic framework that captures the interaction between noncommutative operations and Hom-type twisting maps. More precisely, a noncommutative Hom-anti-pre-Poisson superalgebra $(\mathcal{A}, \circ, \prec, \succ, \alpha)$ consists of a Hom-anti-pre-Lie superalgebra $(\mathcal{A}, \circ, \alpha)$ together with a Hom-anti-dendriform superalgebra $(\mathcal{A}, \prec, \succ, \alpha)$ defined on the same underlying superspace, such that certain compatibility conditions between the operations $\circ$, $\prec$ and $\succ$ are satisfied.
 
This construction allows the extension of classical anti-pre-Poisson algebras into the Hom-type, noncommutative, and superalgebraic context, providing a rich framework for further study in deformation theory, Hom-algebra cohomology, and applications in mathematical physics where graded and twisted structures naturally appear. In particular, it highlights how Hom-anti-pre-Lie superalgebras and Hom-anti-dendriform superalgebras are interwoven within a single coherent algebraic system.\\
 \subsection*{Organization} The paper is organized as follows:\\
 
  In Section \ref{Sec2}, we present the necessary preliminaries and fundamental results that will be used throughout the paper, including Hom-associative superalgebras, Hom-Lie superalgebras, and noncommutative Hom-Poisson superalgebras. Moreover, we introduce the concvepts of Hom-dendriform superalgebras and noncommutative Hom-pre-Poisson superalgebras, accompanied by examples and several related properties. Section \ref{Sec3} is devoted to the introduction of several Hom-type anti-structures in the super setting. In particular, we develop the theory of Hom-anti-associative superalgebras, Hom-anti-dendriform superalgebras, and Hom-anti-pre-Lie superalgebras, together with illustrative examples and constructions arising from anti-super-$\mathcal{O}$-operators. Finally, in Section \ref{Sec4}, we introduce the notion of noncommutative Hom-pre-Poisson superalgebras and noncommutative Hom-anti-pre-Poisson superalgebras, and investigate the compatibility relations between the underlying operations, highlighting their connections with the previously developed Hom-anti-algebraic structures.\\

Throughout this paper, unless otherwise specified, all vector spaces are assumed to be
finite-dimensional over a field $\mathbb{F}$ of characteristic 0.

\section{Basics on noncommutative Hom-(pre-)Poisson superalgebras}\label{Sec2}
In this section, we briefly review Hom-associative and Hom-dendriform superalgebras, as well as Hom-Lie superalgebras and their representations. We also present the definitions of noncommutative Hom-Poisson superalgebras and noncommutative Hom-pre-Poisson superalgebras, together with some examples and fundamental properties. These structures provide the algebraic framework for the results developed in the subsequent sections.
\subsection{Hom-associative and Hom-dendriform superalgebras}
\begin{df}(\cite{Faouzi-Abdenacer})
A \textbf{Hom-associative superalgebra} is a triplet $(\mathcal{A},\mu,\alpha)$ consisting of a $\mathbb{Z}_2$-graded vector space $\mathcal{A}$, an even linear map
 $\alpha:\mathcal{A}\rightarrow\mathcal{A}$, (i.e., $\alpha(\mathcal{A}_i)\subseteq\mathcal{A}_i)$ and an even bilinear map
 $\mu:\mathcal{A} \times \mathcal{A}\rightarrow\mathcal{A}$, (i.e., $\mu(\mathcal{A}_i,\mathcal{A}_j)\subseteq\mathcal{A}_{i+j})$ satisfying the following condition
 for all $x,y,z\in\mathcal{A}$:
 $$ass_{\mu}^\alpha(x,y,z)=0\;(\text{Hom-associativity}),$$
 where $ass_{\mu}^\alpha(x,y,z)=\mu(\mu(x,y),\alpha(z))-\mu(\alpha(x),\mu(y,z)),\;\forall x,y,z\in\mathcal{A}.$
\end{df}
If in addition $\mu$ is super-commutative (i.e., $\mu(x,y)=(-1)^{|x||y|}\mu(y,x)\;\forall\;x,y\in\mathcal{H}(\mathcal{A})$),
the Hom-associative superalgebra $(\mathcal{A},\mu,\alpha)$ is said to be commutative Hom-associative superalgebra.
\begin{rem}
We recover the classical associative superalgebra when $\alpha = id$.
\end{rem}

\begin{exa}(\cite{O-N})
Let $\mathcal{A}=\mathcal{A}_{\overline{0}}\oplus \mathcal{A}_{\overline{1}}$ be a $3$-dimensional $\mathbb{Z}_2$-graded vector space, where $\mathcal{A}_{\overline{0}}=<e_1,e_2>$
 and $\mathcal{A}_{\overline{1}}=<e_3>$. The nonzero products on the basis of $\mathcal A$ given by
 $$\mu(e_1,e_2)=-e_1,\;\mu(e_2,e_2)=e_1+e_2,$$
 and the even linear map $\alpha:\mathcal{A}\rightarrow\mathcal{A}$ defined on the basis of $\mathcal{A}$ by
 $$\alpha(e_1)=-e_1,\;\alpha(e_2)=e_1+e_2\;\text{and}\;\alpha(e_3)=0,$$ defined on $\mathcal{A}$ a structure of Hom-associative superalgebra.
\end{exa}
\begin{df}
 A \textbf{representation} of a Hom-associative superalgebra $(\mathcal A, \mu, \alpha)$  is a quadruplet $(V,\mathfrak l,\mathfrak r,\beta)$ where $V$ is a $\mathbb Z_2$-graded vector space, $\beta\in End_{\mathbb F}(V)$ and $\mathfrak l,\mathfrak r: \mathcal A\rightarrow End_{\mathbb F}(V)$ are three even linear maps
such that the following conditions hold for all $x, y \in\mathcal A:$
\begin{eqnarray}
\beta\mathfrak l(x)=\mathfrak l(\alpha(x))\beta\;\;&\text{and}&\;\;
 \beta\mathfrak r(x)=\mathfrak r(\alpha(x))\beta\label{rAs1}\\
 \mathfrak l(\mu(x, y))\beta&=&\mathfrak l(\alpha(x))\mathfrak l(y)\label{rAs2}\\
 \mathfrak r(\mu(x, y))\beta&=&(-1)^{|x||y|}\mathfrak r(\alpha(y))\mathfrak r(x)\label{rAs3}\\
 \mathfrak l(\alpha(x))\mathfrak r(y)&=&(-1)^{|x||y|}\mathfrak r(\alpha(y))\mathfrak l(x)\label{rAs4}
\end{eqnarray}   
In particular, $(\mathcal A, \mathfrak L,\mathfrak R, \alpha)$ is a representation over $(\mathcal A,\mu,\alpha)$ which is called the adjoint representation, where $\mathfrak L,\mathcal R: \mathcal A\to End_\mathbb F(\mathcal A)$ are two even linear maps defined by 
\begin{equation}\label{adj-rep-H-ass-sup}
\mathfrak L(x)(y)=\mathcal R(y)(x)=\mu(x,y),\; \forall x,y\in\mathcal A.
\end{equation}
\end{df}
Suppose that $(\mathcal{A},\mu,\alpha)$ is a Hom-associative superalgebra.
Let $V$ be a $\mathbb{Z}_2$-graded vector space, $\mathfrak l,\mathfrak r:\mathcal A\to End_{\mathbb F}(V)$ and $\beta\in End_{\mathbb F}(V)$ are three even linear maps. Then $(V,\mathfrak l,\mathfrak r,\beta)$ is a representation of $(\mathcal{A},\mu,\alpha)$ if and only if $(\mathcal{A}\oplus V,\mu_{\mathcal{A}\oplus V},\alpha+\beta)$
is a Hom-associative superalgebra, where $\mu_{\mathcal{A}\oplus V}$ and $(\alpha+\beta)$
are defined for all $x,y\in\mathcal{H}(\mathcal{A}),\;u,v \in\mathcal{H}(V)$ by
\begin{eqnarray}
    \mu_{\mathcal{A}\oplus V}(x+u,y+v)&=&\mu(x,y)+\mathfrak l(x)v+(-1)^{|y||u|}\mathfrak r(y)u,\label{Hom-ass-direct-sum1}\\
    (\alpha+\beta)(x+u)&=&\alpha(x)+\beta(u)\label{Hom-ass-direct-sum2}.
\end{eqnarray}
This Hom-associative superalgebra is called semi-direct product of $(\mathcal{A},\mu,\alpha)$ and $(V,\mathfrak l,\mathfrak r,\beta)$ and denoted by
 $\mathcal{A}\ltimes^{\alpha}_{\mathfrak l,\mathfrak r,\beta}V$ or simply $\mathcal{A}\ltimes_{\mathfrak l,\mathfrak r} V$.\\

\begin{df}\label{Hom-dend-superalg}
 A \textbf{Hom-dendriform superalgebra} is a quadruplet $(\mathcal A,\prec,\succ,\alpha)$ consisting of a $\mathbb Z_2$-graded vector space $\mathcal A$ on which the operations  $\prec,\succ:\mathcal A\otimes\mathcal A\to\mathcal A$ and $\alpha:\mathcal A\to\mathcal A$ are even linear maps satisfying: 
 \begin{align}
(x\prec y)\prec \alpha(z)&=\alpha(x)\prec(y\prec z+y\succ z),\label{cond-Hom-dend1}\\
(x\succ y)\prec\alpha(z)&=\alpha(x)\succ(y\prec z),\label{cond-Hom-dend2}\\
\alpha(x)\succ(y\succ z)&=(x\prec y+x\succ y)\succ\alpha(z),\label{cond-Hom-dend3}
 \end{align}
 for any $x,y,z\in\mathcal H(\mathcal A)$.\\
 We recover the classical dendriform superalgebra when $\alpha=id$. A Hom-dendriform superalgebra $(\mathcal A,\prec,\succ,\alpha)$ is called multiplicative if $\alpha$ is an algebra endomorphism that is $$\alpha(x\prec y)=\alpha(x)\prec\alpha(y)\;\;\text{and}\;\;\alpha(x\succ y)=\alpha(x)\succ\alpha(y),\;\forall x,y\in\mathcal H(\mathcal A).$$
\end{df}
\begin{exa}
Let $\mathcal{A} = \mathcal{A}_{\overline{0}} \oplus \mathcal{A}_{\overline{1}}$ be a $3$-dimensional superspace, where
$\mathcal{A}_{\overline{0}}=<e_0>$ and $\mathcal{A}_{\overline{1}}=<e_1,e_2>$. We define on the basis of $\mathcal{A}$, the following linear map
$$\alpha(e_0)=e_0\;,\;\alpha(e_1)=-e_1\;,\;\alpha(e_2)=-e_2,$$
and the following nonzero products
$$ e_2\succ e_2 =-e_0\;,\; e_0\succ e_2=\frac{1}{4}e_1\;,\; e_2\succ e_0=-\frac{1}{4}e_1$$
$$e_2\prec e_2 =2e_0\;,\; e_0\prec e_2=\frac{1}{2}e_1\;,\; e_2\prec e_0=e_1 $$
Then the quadruplett $(\mathcal{A},\prec,\succ,\alpha)$ is a Hom-dendriform superalgebra.
\end{exa}
\begin{prop}\label{Hom-dend-to-Hom-ss}
 Let $(\mathcal A,\prec,\succ,\alpha)$ be a Hom-dendriform superalgebra. Let $\star:\mathcal A\otimes\mathcal A \to\mathcal A$ be a linear map
defined for $x, y \in\mathcal A$ by
\begin{equation}\label{comp-hom-ass-of-Hom-dend}
x\star y=x\prec y+x\succ y.    
\end{equation}
Then $(\mathcal A,\star,\alpha)$ is a Hom-associative superalgebra.
\end{prop}
\begin{proof}Let $x,y,z\in\mathcal H(\mathcal A)$, we have
\begin{align*}
\alpha (x)\star (y\star z)&= \alpha (x)\star (y\prec z+ y \succ z),\\
\ &= \alpha (x)\prec (y\prec z+ y \succ z)+\alpha (x)\succ (y\prec z+ y \succ z),\\
\ &= (x\prec y)\prec \alpha (z)+\alpha (x)\succ (y\prec z)+\alpha (x)\succ( y \succ z),\\
\ &= (x\prec y)\prec \alpha (z)+(x\succ y)\prec \alpha (z)+(x \prec y+x\succ y) \succ \alpha (z),\\
\ &= (x\prec y+x\succ y)\prec \alpha (z)+(x \prec y+x\succ y) \succ \alpha (z),\\
\ &= (x\star y)\prec \alpha (z)+(x\star y) \succ \alpha (z),\\
\ &= (x\star y)\star \alpha (z).
\end{align*}

\end{proof}
\begin{thm}\label{dend-twist-thm}
Let $(\mathcal A,\prec,\succ)$ be a dendriform superalgebra and $\alpha:\mathcal A\to\mathcal A$ be an even algebra endomorphism. Then, $(\mathcal A,\prec_\alpha,\succ_\alpha,\alpha)$ where $\prec_\alpha,\succ_\alpha:\mathcal A\otimes\mathcal A\to\mathcal A$ are defined by:
\begin{align}
 x\prec_\alpha y&=\alpha(x)\prec\alpha(y),\label{con-twist-dendr1}\\
 x\succ_\alpha y&=\alpha(x)\succ\alpha(y).
\end{align}
is a multiplicative Hom-dendriform superalgebra.    
\end{thm}
\begin{proof}
It's easy to show that, 
$$
 \alpha(x\prec_\alpha y)=\alpha(x)\prec_\alpha \alpha(y),\;\;\;\text{and}\;\;\;
\alpha(x\succ_\alpha y)=\alpha(x)\succ_\alpha \alpha(y),\;\forall x,y\in\mathcal{H}(\mathcal A).
$$
Let $x,y,z\in\mathcal H(\mathcal A)$, then we have:
\begin{align*}
(x\prec_\alpha y)\prec_\alpha\alpha(z)&=(\alpha^{2}(x)\prec\alpha^{2}(y))\prec\alpha^{2}(z)\\
&=\alpha^{2}((x\prec y)\prec z)\\
&=\alpha^{2}(x\prec(y\prec z+y\succ z))\\
&=\alpha(x)\prec_\alpha(y\prec_\alpha z+y\succ_\alpha z)
\end{align*}
\begin{align*}
(x\succ_\alpha y)\prec_\alpha\alpha(z)&=(\alpha^{2}(x)\succ\alpha^{2}(y))\prec\alpha^{2}(z)\\
&=\alpha^{2}((x\succ y)\prec z)\\
&=\alpha^{2}(x\succ(y\prec z)) \\
&=\alpha(x)\succ_\alpha(y\prec_\alpha z)
\end{align*}
\begin{align*}
\alpha(x)\succ_\alpha(y\succ_\alpha z)&=\alpha^{2}(x)\succ(\alpha^{2}(y)\succ\alpha^{2}(z))\\
&=\alpha^{2}(x\succ(y\succ z))\\
&=\alpha^{2}((x\prec y+x\succ y)\succ z) \\
&=(x\prec_\alpha y+x\succ_\alpha y)\succ_\alpha\alpha(z),
\end{align*}
which gives the result.
\end{proof}
\begin{prop}\label{equiv-Hom-dend-Hom-ass}
Let $\mathcal A$ be a $\mathbb Z_2$-graded vector space equipped with two even bilinear operations $\prec,\succ:\mathcal A\times\mathcal A\to\mathcal A$ and an even linear map $\alpha:\mathcal A\to\mathcal A$. Then the following conditions are equivalent:
\begin{enumerate}
\item $(\mathcal A,\prec,\succ,\alpha)$ is a Hom-dendriform superalgebra.
\item The triplet $(\mathcal A,\star,\alpha)$ where $\star$ is defined by Eq. \eqref{comp-hom-ass-of-Hom-dend} is a  Hom-associative superalgebra and $(\mathcal A,\mathfrak L_\succ,\mathfrak R_\prec,\alpha)$ is a representation of $(\mathcal A,\star,\alpha)$,
\item $(\mathcal A\oplus\mathcal A,\bullet,\alpha^{\otimes2})$ is a Hom-associative superalgebra where $\bullet$ and $\alpha^{\otimes2}$ are defined by 
\begin{align}
\alpha^{\otimes2}(x,a)&=(\alpha(x),\alpha(a)),\label{Hom-assoc-direct-sum-Hom-dend1}\\
(x,a)\bullet(y,b)&=(x\prec y+x\succ y,x\succ b+a\prec y),\label{Hom-assoc-direct-sum-Hom-dend2}
\end{align}
for any $x,y,a,b\in\mathcal H(\mathcal A)$.
\end{enumerate}
\end{prop}
\begin{proof}
$(1)\Leftrightarrow (2)$ By Proposition \ref{Hom-dend-to-Hom-ss}, the triplet $(\mathcal A,\star,\alpha)$  is a  Hom-associative superalgebra.\\
For any $x,y,z\in\mathcal H(\mathcal A)$, we have:
\begin{align*}
\mathfrak L_\succ(x\star y)\alpha(z)&=(x\star y)\succ\alpha(z)\\
&=(x\prec y+x\succ y)\succ\alpha(z)\\
&=\alpha(x)\succ(y\succ z)\\
&=\mathfrak L_\succ(\alpha(x))\mathfrak L_\succ(y)z.
\end{align*}
Similarly we have:
\begin{align*}
\mathfrak R_\prec(x\star y)\alpha(z)&=(-1)^{|z|(|x|+|y|)}\alpha(z)\prec (x\star y)\\
&=(-1)^{|z|(|x|+|y|)}(z\prec x)\prec \alpha(y) \\
&=(-1)^{|x||y|}\mathfrak R_\prec(\alpha(y))\mathfrak R_\prec(x)z,
\end{align*}
and
\begin{align*}
\mathfrak L_\succ(\alpha(x))\mathfrak R_\prec(y)z&=(-1)^{|y||z|}\alpha(x)\succ(z\prec y)\\
&=(-1)^{|y||z|}(x\succ z)\prec\alpha(y) \\
&=(-1)^{|x||y|}\mathfrak R_\prec(\alpha(y))\mathfrak L_\succ(x)z.
\end{align*}
 Then $(\mathcal A,\mathfrak L_\succ,\mathfrak R_\prec,\alpha)$ is a representation of $(\mathcal A,\star,\alpha)$.\\
 
 $(2)\Leftrightarrow(3)$ Let $x,y,z,a,b,c\in\mathcal H(\mathcal A)$, then:
 \begin{align*}
 ((x,a)\bullet(y,b))\bullet\alpha^{\otimes2}(z,c)=&(x\prec y+x\succ y,x\succ b+a\prec y)\bullet(\alpha(z),\alpha(c))\\
 =&\Big((x\prec y+x\succ y)\prec\alpha(z)+(x\prec y+x\succ y)\succ \alpha(z),\\
&(x\prec y+x\succ y)\succ \alpha(c)+(x\succ b+a\prec y)\prec\alpha(z)\Big)\\
=&\Big((x\star y)\star\alpha(z),\mathfrak L_\succ(x\star y)\alpha(c)+(-1)^{|x||z|}\mathfrak R_\prec(\alpha(z))\mathfrak L_\succ(x)b\\
&+(-1)^{|y||z|}\mathfrak R_\prec(\alpha(z))\mathfrak R_\prec(y)a\Big), \\
\end{align*}
and\\ 

\begin{align*}
\alpha^{\otimes2}(x,a)\bullet ((y,b)\bullet(z,c))=&
(\alpha(x),\alpha(a))\bullet(y\prec z+y\succ z,y\succ c+b\prec z) \\
=&\Big(\alpha(x)\prec(y\prec z+y\succ z)+\alpha(x)\succ(y\prec z+y\succ z),\\
&\alpha(x)\succ(y\succ c+b\prec z)+\alpha(a)\prec(y\prec z+y\succ z)\Big)\\
=&\Big((x\star y)\star\alpha(z),\mathfrak L_\succ(\alpha(x))\mathfrak L_\succ(y)c+\mathfrak L_\succ(\alpha(x))\mathfrak R_\prec(z)b+\mathfrak R_\prec(y\star z)\alpha(a)\Big).
  \end{align*}
  Then, $(\mathcal A\oplus\mathcal A,\bullet,\alpha^{\otimes2})$ forms a Hom-associative superalgebra if and only if the following identity holds: 
  $$((x,a)\bullet(y,b))\bullet\alpha^{\otimes2}(z,c)=\alpha^{\otimes2}(x,a)\bullet ((y,b)\bullet(z,c)).$$
  This condition translates into the pair of equations:
  $$(x\star y)\star\alpha(z)=\alpha(x)\star(y\star z)$$ and 
  \begin{align*}
  &~\mathfrak L_\succ(x\star y)\alpha(c)+(-1)^{|x||z|}\mathfrak R_\prec(\alpha(z))\mathfrak L_\succ(x)b+(-1)^{|y||z|}\mathfrak R_\prec(\alpha(z))\mathfrak R_\prec(y)a\\
  =~&\mathfrak L_\succ(\alpha(x))\mathfrak L_\succ(y)c+\mathfrak L_\succ(\alpha(x))\mathfrak R_\prec(z)b+\mathfrak R_\prec(y\star z)\alpha(a).
  \end{align*}
  We conclude that $(\mathcal A,\star,\alpha)$ is a Hom-associative superalgebra, and by identification, the quadruplet $(\mathcal A,\mathfrak L_\succ,\mathfrak R_\prec,\alpha)$ constitutes a representation of $(\mathcal A,\star,\alpha)$.\\
  
\end{proof}
\subsection{Hom-Lie superalgebras}
\begin{df}(\cite{Faouzi-Abdenacer})
A \textbf{Hom-Lie superalgebra} is a triplet $(\mathcal{A}, [\cdot,\cdot],\alpha)$ consisting of a $\mathbb{Z}_2$-graded vector space $\mathcal{A}$, an even bilinear map $[\cdot,\cdot] : \mathcal{A}\otimes \mathcal{A} \longrightarrow \mathcal{A},~~( \;[\mathcal{A}_i,\mathcal{A}_j]\subseteq \mathcal{A}_{i+j},
~~\forall~~i,j\in \mathbb{Z}_2\;)$ and an even linear map  $\alpha:\mathcal{A}\rightarrow\mathcal{A}$  satisfying:
\begin{eqnarray}
 \label{H-skwesym}
 &&[x,y] = -(-1)^{|x||y|}[y,x]~~ \text{( super-skew-symmetry)},\\
\label{H-sJ}
&&  [\alpha(x),[y,z]]=[[x,y],\alpha(z)]+(-1)^{|x||y|}[\alpha(y),[x,z]]~~ \text{(Hom super-Jacobi identity),}
\end{eqnarray}
$\forall\ x,y,z \in \mathcal{H}(\mathcal{A})$.\\
We recover the classical Lie superalgebra  when $\alpha=id$.
\end{df}
\begin{exa}(\cite{O-N})
Let $\mathcal{A}=\mathcal{A}_{\overline{0}}\oplus \mathcal{A}_{\overline{1}}$ be a $2$-dimensional $\mathbb{Z}_2$-
graded vector space, where $\mathcal{A}_{\overline{0}}=<e_1>$
 and $\mathcal{A}_{\overline{1}}=<e_2>$. We consider the nonzero product given by
$$[e_2,e_2]=2\lambda e_1,$$
and the even linear map $\alpha:\mathcal{A}\rightarrow\mathcal{A}$ defined on the basis of $\mathcal{A}$ by $$\alpha(e_1)=\lambda e_1\;\text{and}\;\alpha(e_2)=\lambda e_2,$$
where $\lambda\in\mathbb{C}$. The triplet $(\mathcal{A},[\cdot,\cdot],\alpha)$ is a Hom-Lie superalgebra.
\end{exa}
\begin{thm}(\cite{Faouzi-Abdenacer})\label{Lie-twist-thm}
 Let $(\mathcal A,[\cdot,\cdot])$ be a Lie superalgebra and $\alpha:\mathcal A\to\mathcal A$ be an even algebra endomorphism. Define the linear map $[\cdot,\cdot]_\alpha:\mathcal A\otimes\mathcal A\to$ by
 \begin{equation}\label{exp-twist-Lie-sup}
  [x,y]_\alpha=[\alpha(x),\alpha(y)],\;\forall x,y\in\mathcal A.   
 \end{equation}
 Then, $(\mathcal A,[\cdot,\cdot]_\alpha,\alpha)$ is a multiplicative Hom-Lie superalgebra.
\end{thm}
\begin{prop}(\cite{Faouzi-Abdenacer})\label{Hom-Lie-super-twist}
Let $(\mathcal{A},\mu,\alpha)$ be a Hom-associative superalgebra, then $(\mathcal{A},[\cdot,\cdot]_\mu,\alpha)$ is a Hom-Lie superalgebra, where
$$[x,y]_\mu=\mu(x,y)-(-1)^{|x||y|}\mu(y,x),$$
for all $x,y\in\mathcal{H}(\mathcal{A})$.
\end{prop}
As a direct consequence of Proposition \eqref{Hom-dend-to-Hom-ss}, we obtain the following result.
\begin{cor}
Let $(\mathcal A,\prec,\succ,\alpha)$ be a Hom-dendriform superalgebra. Then, the linear map $[\cdot,\cdot]_{\prec,\succ}:\mathcal A\otimes\mathcal A\to\mathcal A$ defined by:
\begin{equation}\label{ind-Hom-Lie-hom-dendr}
[x,y]_{\prec,\succ}=x\prec y-y\prec x+x\succ y-y\succ x,\;\forall x,y\in\mathcal H(\mathcal A),    
\end{equation}
defines on $\mathcal A$ a Hom-Lie superalgebra structure.
\end{cor}
\begin{df}(\cite{Fawzi-Makhlouf-Saadaoui})
A \textbf{representation} of a Hom-Lie superalgebra $(\mathcal{A},[\cdot,\cdot],\alpha)$ on a $\mathbb{Z}_2$-graded vector space $V$ with respect to $\beta\in End_{\mathbb F}(V)$ is an even linear map $\rho:\mathcal{A}\to End_{\mathbb F}(V)$ such that for all $x,y\in\mathcal{H}(\mathcal A)$, the following equalities are satisfied:
\begin{eqnarray}
\rho(\alpha(x))\beta&=&\beta\rho(x),\label{repres-Hom-Lie1}\\
\rho([x,y])\beta&=&\rho(\alpha(x))\rho(y)-(-1)^{|x||y|}\rho(\alpha(y))\rho(x).\label{repres-Hom-Lie2}
\end{eqnarray}
This type of representation of a Hom-Lie superalgebra is typically denoted by $(V, \rho, \beta)$.
\end{df}
\begin{exa}
For any $x\in\mathcal{H}(\mathcal A)$, the linear map $\mathfrak{ad}:\mathcal A\to gl(\mathcal A);\;x\mapsto \mathfrak{ad}(x)$ (or $\mathfrak{ad}_x$) defined by
$$\mathfrak{ad}_x(y)=[x,y],\;\forall y\in\mathcal{H}(\mathcal A),$$
defines a representation on $\mathcal A$ called \textbf{adjoint representation}.
\end{exa}
\begin{prop}(\cite{O-N})\label{sumdirecthomLie}
Let $(\mathcal{A}, [\cdot,\cdot ],\alpha)$ be a Hom-Lie superalgebra and $(V,\beta)$ be a Hom-module. Let
$\rho:\mathcal{A} \to gl(V )$ be an even linear map. The triplet $(V,\rho,\beta)$ is a representation of
$(\mathcal{A}, [\cdot,\cdot ],\alpha)$ if and only if the direct sum  $\mathcal{A}\oplus V $ of $\mathbb{Z}_2$-graded vector spaces $\mathcal{A}$ and $V$
turns into a Hom-Lie superalgebra by defining the linear map \eqref{Hom-ass-direct-sum2} and the following multiplication
\begin{equation}
    [x+u,y+v]_{\mathcal{A}\oplus V}=[x,y]+\rho(x)v-(-1)^{|y||u|}\rho(y)u\label{Hom-Lie-direct-sum},
\end{equation}
called semi-direct product of the Hom-Lie superalgebra $(\mathcal A,[\cdot,\cdot],\alpha)$ and $V$ which simply denoted by $\mathcal A\ltimes_{\rho} V$.
\end{prop}

\subsection{Noncommutative Hom-Poisson superalgebras and noncommutative Hom-pre-Poisson superalgebras}
\begin{df}(\cite{O-N})
A \textbf{non-commutative Hom-Poisson superalgebra} is a quadruplet $(\mathcal{A}, [\cdot, \cdot], \mu, \alpha)$ consisting of
\begin{enumerate}
\item A Hom-Lie superalgebra $(\mathcal{A}, [\cdot, \cdot],\alpha)$,
\item A Hom-associative superalgebra $(\mathcal{A},\mu, \alpha)$,
\item The Hom-Leibniz superalgebra identity
\begin{equation}\label{Homleibnizident}
[\alpha(x),\mu(y,z)]=\mu([x,y],\alpha(z))+(-1)^{|x||y|}\mu(\alpha(y),[x,z])
\end{equation}
is satisfied for all $x,y,z\in\mathcal{H}(\mathcal{A})$.
\end{enumerate}
\end{df}
A Hom-Poisson superalgebra is a non-commutative Hom-Poisson superalgebra
 $(\mathcal{A}, [\cdot, \cdot], \mu, \alpha)$ in
which $\mu$ is super-commutative.\\

\begin{exa}(\cite{O-N})
Let $\mathcal{A} = \mathcal{A}_{\overline{0}} \oplus \mathcal{A}_{\overline{1}}$ be a $3$-dimensional superspace, where
$\mathcal{A}_{\overline{0}}=<e_1,e_2>$ and $\mathcal{A}_{\overline{1}}=<e_3>$. We define on the basis of $\mathcal{A}$, the following linear map
$$\alpha(e_1)=e_1\;,\;\alpha(e_2)=e_1+e_2,$$
and the following nonzero products
$$\mu(e_1,e_2)=e_1\;,\;\mu(e_2,e_2)=e_1+e_2,$$
$$[e_1,e_2]=e_1.$$
Then the quadruplett $(\mathcal{A},[\cdot,\cdot],\mu,\alpha)$ is a non-commutative Hom-Poisson superalgebra.
\end{exa}


\begin{rem}
If $(\mathcal{A}, [\cdot, \cdot], \mu, \alpha)$ is a Hom-Poisson superalgebra, then the identity \eqref{Homleibnizident} is equivalent to
\begin{equation}\label{equiv-Homleibnizident}
[\alpha(x),\mu(y,z)]=\mu([x,y],\alpha(z))+(-1)^{|y||z|}\mu([x,z],\alpha(y)).
\end{equation}
\end{rem}
If $\alpha$ is a Poisson algebra endomorphism, that is,
\[
\alpha([x,y])=[\alpha(x),\alpha(y)] \quad \text{and} \quad 
\alpha(\mu(x,y))=\mu(\alpha(x),\alpha(y))
\]
for all $x,y \in \mathcal{H}(\mathcal{A})$, then $\mathcal{A}$ is called a \emph{\textbf{multiplicative Hom-Poisson superalgebra}}. Moreover, if the twisting map $\alpha$ is bijective, then $\mathcal{A}$ is said to be a \emph{\textbf{regular Hom-Poisson superalgebra}}.
\begin{df}\label{repr-noncom-Hom-Poiss}
Let $(\mathcal A,[\cdot,\cdot],\mu,\alpha)$ be a noncommutative Hom-Poisson superalgebra. A representation of $\mathcal A$ is a quintuple $(V,\rho,\mathfrak l,\mathfrak r,\beta)$ such that $(V,\rho,\beta)$ is a representation of the Hom-Lie superalgebra $(\mathcal A,[\cdot,\cdot],\alpha)$ and $(V,\mathfrak l,\mathfrak r,\beta)$ is a representation of the Hom-associative superalgebra $(\mathcal A,\mu,\alpha)$ satisfying for any $x,y\in\mathcal{H}(\mathcal A)$ the following equalities:
\begin{align}
\mathfrak l([x,y])\beta&=\rho(\alpha(x))\mathfrak l(y)-(-1)^{|x||y|}\mathfrak l(\alpha(y))\rho(x),\label{cond-rep-noncom-Hom-Pois1}\\   
\mathfrak r([x,y])\beta&=\rho(\alpha(x))\mathfrak r(y)-(-1)^{|x||y|}\mathfrak r(\alpha(y))\rho(x),\label{cond-rep-noncom-Hom-Pois2}\\
\rho(\mu(x,y))\beta&=\mathfrak l(\alpha(x))\rho(y)+(-1)^{|x||y|}\mathfrak r(\alpha(y))\rho(x).\label{cond-rep-noncom-Hom-Pois3}
\end{align}
\end{df}
\begin{exa}
 Let $(\mathcal A,[\cdot,\cdot],\mu,\alpha)$ be a noncommutative Hom-Poisson superalgebra. Then $(\mathcal A,\mathfrak{ad},\mathfrak L,\mathfrak R,\alpha)$ is a representation of $\mathcal A$, which is also called the \textbf{regular representation} of $\mathcal A$.   
\end{exa}
\begin{df}
 A \textbf{Hom-pre-Lie superalgebra} is a triplet $(\mathcal A,\circ,\alpha)$ consisting of a $\mathbb Z_2$-graded vector space $\mathcal A$ endowed with an even bilinear map $\circ:\mathcal A\times\mathcal A\to\mathcal A$ and an even linear map $\alpha:\mathcal A\to\mathcal A$ satisfying $\mathfrak{ass}(x,y,z)=(-1)^{|x||y|}\mathfrak{ass}(y,x,z)$, that is,
 \begin{equation}\label{cond-Hom-pre-Lie-sup}
 (x\circ y)\circ\alpha(z)-\alpha(x)\circ(y\circ z)=(-1)^{|x||y|}(y\circ x)\circ\alpha(z)-\alpha(y)\circ(x\circ z),  
 \end{equation}
 for all $x,y,z\in\mathcal H(\mathcal A)$.
\end{df}
\begin{exa}\label{exam-Hom-pre-Lie-sup}
Let $\mathcal A=\mathfrak{osp}(1|2)
=\mathcal A_{\overline 0}\oplus\mathcal A_{\overline 1},
$
where $\mathcal A_{\overline 0}=\langle H,E,F\rangle$ and $\mathcal A_{\overline 1}=\langle X,Y\rangle
$.
Define a linear map
$\alpha:\mathcal A\to\mathcal A$ by
\[
\alpha(H)=H,\quad
\alpha(E)=\frac{1}{2} E,\quad
\alpha(F)=2F,\quad
\alpha(X)= X,\quad
\alpha(Y)=Y,
\]

\medskip

\noindent

and a bilinear map $\circ:\mathcal A\times\mathcal A\to\mathcal A$
by the following nonzero products:
\[
\begin{aligned}
&H\circ E=2E,\qquad H\circ F=-2F,\qquad E\circ F=H,\\
&H\circ X=X,\qquad H\circ Y=-Y,\qquad E\circ Y=X,\qquad F\circ X=Y,\\
&X\circ X=2E,\qquad Y\circ Y=-2F,\qquad X\circ Y=H.
\end{aligned}
\]
Then $\alpha$ is an algebra endomorphism on the pre-Lie superalgebra $(\mathfrak{osp}(1|2),\circ)$ and  $(\mathfrak{osp}(1|2),\circ_\alpha,\alpha)$ is a Hom-pre-Lie superalgebra, where $\circ_\alpha$ is defined by $$x\circ_\alpha y=\alpha(x\circ y),\;\forall x,y\in\mathcal A.$$
\end{exa}
\begin{prop}
 Let $(\mathcal A,\prec,\succ,\alpha)$ be a Hom-dendriform superalgebra. Then, the bilinear product $\circ:\mathcal A\times\mathcal A\to\mathcal A$ defined by
 $$x\circ y=x\prec y-(-1)^{|x||y|}y\succ x,\;\forall x,y\in\mathcal H(\mathcal A),$$ define on $\mathcal A$ a Hom-pre-Lie superalgebra structure.
\end{prop}
\begin{proof}
Straightforward.
\end{proof}
\begin{df}
A \textbf{noncommutative Hom-pre-Poisson superalgebra} is a quintuple $(A,\circ,\prec,\succ,\alpha)$ such that $(A,\circ,\alpha)$ is a Hom-pre-Lie superalgebra and $(A,\prec,\succ,\alpha)$ is a Hom-dendriform superalgebra satisfying the following compatibility conditions :
    \begin{align}
(x\circ y-(-1)^{|x||y|}y \circ x)\succ \alpha(z)&=\alpha(x)\circ(y\succ z)-(-1)^{|x||y|}\alpha(y)\succ(x \circ z),\label{cond-noncomm-Hom-pre-Poisson1}\\
\alpha(x)\prec(y\circ z -(-1)^{|y||z|} z \circ y)&=(-1)^{|x||y|} \alpha(y) \circ (x\prec z)-(-1)^{|x||y|}(y\circ x)\prec \alpha(z), \label{cond-noncomm-Hom-pre-Poisson2}\\
(x\succ y + x\prec y)\circ\alpha(z)&=(-1)^{|y||z|} (x\circ z)\prec \alpha(y)+\alpha(x) \succ(y\circ z).\label{cond-noncomm-Hom-pre-Poisson3}
\end{align}
\end{df}
\begin{rem}
When $\alpha=id$, we recover the noncommutative pre-Poisson superalgebras structures.    
\end{rem}
\begin{exa}\label{ex-noncom-H-pre-Pois}
Let us consider the Hom-pre-Lie superalgebra on $\mathfrak{osp}(1|2)$ given by Example \ref{exam-Hom-pre-Lie-sup}.  



Define two bilinear maps $\prec,\succ:\mathfrak{osp}(1|2)\times\mathfrak{osp}(1|2)\to\mathfrak{osp}(1|2) $ by:


\[
\begin{aligned}
&H\prec E=2E,\qquad E\prec F=H,\\
&H\prec X=X,\qquad E\prec Y=X,\\
&X\prec Y=H,\qquad X\prec X=2E.
\end{aligned}
\]

\medskip
\noindent
\[
\begin{aligned}
&E\succ H=-2E,\qquad F\succ E=-H,\\
&X\succ H=-X,\qquad Y\succ E=-X,\\
&Y\succ X=-H,\qquad Y\succ Y=2F.
\end{aligned}
\]
All other products are zero.
\medskip

\noindent
Define
\[
x\prec_\alpha y=\alpha(x\prec y),\qquad
x\succ_\alpha y=\alpha(x\succ y),
\quad \forall x,y\in\mathcal A.
\]

\medskip

Then $(\mathfrak{osp}(1|2),\circ_\alpha,\prec_\alpha,\succ_\alpha,\alpha)$ is a non-commutative Hom-pre-Poisson superalgebra.   
\end{exa}
\section{Hom-anti-associative superalgebras, Hom-anti-dendriform superalgebras and Hom-anti-pre-Lie superalgebras }\label{Sec3}

In this section, we introduce the Hom-case of some anti-algebraic structures such as, Hom-anti-associative superalgebras, Hom-anti-dendriform superalgebras and Hom-anti-pre-Lie superalgebras for which we give some examples and related results associated to these structures.\\

\subsection{Hom-anti-associative superalgebras}
\hspace{1cm}\\

Let us consider a Hom-superalgebra $(\mathcal A,\mu,\alpha)$, that is a $\mathbb Z_2$-graded vector space $\mathcal A$ with an even bilinear map $\mu$ and
an even linear map $\alpha$. Define the trilinear map $\mathfrak{aas}^\alpha_\mu:\mathcal A\times\mathcal A\times\mathcal A\to\mathcal A$ called \textbf{Hom-anti-associator} of $\mu$ by
\begin{equation}\label{exp-H-aas}
\mathfrak{aas}^\alpha_\mu(x,y,z)=\mu(\alpha(x),\mu(y,z))+\mu(\mu(x,y),\alpha(z)),\;\forall x,y,z\in\mathcal A.    
\end{equation}
If $\alpha=id$, we obtain \textbf{anti-associator} associated to $\mathcal A$ denoted by $\mathfrak{aas}_\mu$ and $(\mathcal A,\mu,\alpha)$ is called super-commutative (resp. super-anti-commutative) if $\mu$ is super-commutative (resp. super-anti-commutative), that is $$\mu(x,y)=(-1)^{|x||y|}\mu(y,x),\forall x,y\in\mathcal A,$$
(resp. $\mu(x,y)=-(-1)^{|x||y|}\mu(y,x),\;\forall x,y\in\mathcal A).$
\begin{df}
A  \textbf{Hom-anti-associative superalgebra} is a $\mathbb Z_2$-graded vector space $\mathcal A$ equipped with an even bilinear map $\mu:\mathcal A\times\mathcal A\to\mathcal A$ and an even linear map $\alpha:\mathcal A\to\mathcal A$ such that 
\begin{equation}\label{cond-noncom-anti-ass}
\mathfrak{aas}^\alpha_\mu(x,y,z)=0,\;\forall x,y,z\in\mathcal A.
\end{equation}
A Hom-anti-associative superalgebra $(\mathcal A,\mu,\alpha)$ is called:
\begin{enumerate}
\item \textbf{regular} if $\alpha$ is bijective. \item \textbf{multiplicative} if $\alpha$ is an algebra endomorphism that is, $\alpha\mu(x, y)=\mu(\alpha(x),\alpha(y)),$ for any $x,y\in\mathcal A$.
\end{enumerate}
\end{df}
\begin{rem}\
We recover the classical \textbf{anti-associative superalgebra} when $\alpha=id$.   
\end{rem}
\begin{exa}
Let $\mathcal{A}=\mathcal{A}_{\overline{0}}\oplus \mathcal{A}_{\overline{1}}$ be a $3$-dimensional $\mathbb{Z}_2$-
graded vector space, where $\mathcal{A}_{\overline{0}}=<e_0>$
 and $\mathcal{A}_{\overline{1}}=<e_1,e_2>$. The nonzero products given by
 $$\mu(e_1,e_2)=-\mu(e_2,e_1)=e_0$$
 and the even linear map $\alpha:\mathcal{A}\rightarrow\mathcal{A}$ defined on the basis of $\mathcal{A}$ by
 $$\alpha(e_0)=2e_0,\;\alpha(e_1)=e_1\;\text{and}\;\alpha(e_2)=-e_2,$$ defined on $\mathcal{A}$ a structure of Hom-anti-associative superalgebra.
\end{exa}
\begin{prop}\label{Prop-twis-Hom-ant-ass}
Let $(\mathcal A,\mu)$ be an anti-associative superalgebra and $\alpha:\mathcal A\to\mathcal A$ be an algebra endomorphism of $\mathcal A$. Define the binary operation $\mu_\alpha:\mathcal A\times\mathcal A\to\mathcal A$ by 
\begin{equation}\label{twist-Hom-anti-asso}
\mu_\alpha(x,y)=\mu(\alpha(x),\alpha(y)),\;\forall x,y\in\mathcal A.    
\end{equation}
Then, $(\mathcal A,\mu_\alpha,\alpha)$ is a Hom-anti-associative superalgebra.
\end{prop}
\begin{proof}
Let $x,y,z\in\mathcal A$. By using Eq. \eqref{cond-noncom-anti-ass} and the fact that $\alpha$ is an algebra endomorphism, we have
\begin{align*}
\mathfrak{aas}_{\mu_\alpha}^\alpha(x,y,z)&=\mu_\alpha(\alpha(x),\mu_\alpha(y,z))+ \mu_\alpha(\mu_\alpha(x,y),\alpha(z))\\&= \mu(\alpha^2(x),\alpha\mu(\alpha(y),\alpha(z)))+\mu(\alpha\mu(\alpha(x),\alpha(y)),\alpha^2(z))\\&=\alpha^2\big(\mu(x,\mu(y,z))+\mu(\mu(x,y),z)\big)\\&=\alpha^2\mathfrak{aas}_\mu(x,y,z)\\&=0,  
\end{align*}
which gives the proof.
\end{proof}
\begin{exa}
Let $\mathcal{A}=< e_0,e_1,e_2>$ be a $3$-dimensional $\mathbb{Z}_2$- graded vector space where $\mathcal{A}_{\overline{0}}=<e_0>$
 and $\mathcal{A}_{\overline{1}}=<e_1,e_2>$. Define the bilinear map
$\mu :\mathcal{A}\times \mathcal{A}\rightarrow \mathcal{A}$ by 
$$ \mu(e_1,e_1)=e_0 ,\;\mu(e_0,e_1)=e_2,\mu(e_1,e_0)=-e_2.$$
Then $(\mathcal{A},\mu)$ is an anti-associative superalgebra. 
Define the linear map $\alpha:\mathcal{A}\rightarrow\mathcal{A}$ by $$\alpha(e_0)=e_0,\;\alpha(e_1)=-e_1\;\text{and}\;\alpha(e_2)=-e_2.$$
Then, by Proposition \ref{Prop-twis-Hom-ant-ass}, the triplet  $(\mathcal{A},\mu_\alpha,\alpha)$ define a Hom-anti-associative superalgebra.
\end{exa}
\begin{exa}\label{ex-twis-H-ant-ass}
Let $\mathfrak{osp}(1,2)=\mathcal A_{\overline{0}} \oplus\mathcal A_{\overline{1}}$  \ be  the superspace where
$\mathcal A_{\overline 0}$ is generated by:
$$ H=\left(
  \begin{array}{ccc}
  1 & 0& 0 \\
  0 &0 & 0 \\
    0 & 0 & -1\\
  \end{array}\right), \ \ X=\left(
  \begin{array}{ccc}
  0 & 0 & 1 \\
  0 & 0 & 0 \\
  0 & 0 & 0\\
  \end{array}
\right),\ \ Y=\left(
  \begin{array}{ccc}
  0 & 0& 0 \\
  0 & 0 & 0 \\
  1 & 0 & 0\\
  \end{array}
\right), $$ and $\mathcal A_{\overline 1}$ is generated by
$$ F=\left(
  \begin{array}{ccc}
  0 & 0 & 0 \\
  1 & 0 & 0 \\
  0 & 1 & 0\\
  \end{array}
\right) , \ \ G=\left(
  \begin{array}{ccc}
  0 & 1& 0 \\
  0 & 0 & -1 \\
  0 & 0 & 0\\
  \end{array}
\right) .$$

Define a bilinear map
\[
\mu:\mathfrak{osp}(1|2)\times\mathfrak{osp}(1|2)\to\mathfrak{osp}(1|2)
\]
by the following nonzero products:
\[
\begin{aligned}
\mu(H,H) &=X,\;\;
\mu(F,G) &=-\mu(G,F)= Y.
\end{aligned}
\]
All other products are zero. Then $(\mathfrak{osp}(1|2),\mu)$ is an anti-associative superalgebra.\\
Define a linear map
$$\alpha:\mathfrak{osp}(1|2)\to\mathfrak{osp}(1|2)$$ on the homogeneous basis by
$$
\alpha(H) =2H,\;\;\alpha(X)=4X,\;\;\alpha(Y)=3Y,\;\;\alpha(F)=F,\;\;\alpha(G)=3G.
$$
Then, by Proposition \ref{Prop-twis-Hom-ant-ass}, the triplet $(\mathfrak{osp}(1|2),\mu_\alpha,\alpha)$ is a Hom-anti-associative superalgebra.
   
\end{exa}
\subsection{Hom-anti-dendriform superalgebras}
\begin{df}
 A \textbf{Hom-anti-dendriform superalgebra} is a quadruplet $(\mathcal A,\prec,\succ,\alpha)$ consisting of a $\mathbb Z_2$-graded vector space $\mathcal A$, two even bilinear operations $\prec,\succ:\mathcal A\times\mathcal A\to\mathcal A$ and an even linear map $\alpha:\mathcal A\to\mathcal A$ satisfying
 \begin{eqnarray}
\alpha(x)\prec(y\prec z)=-(x\cdot y)\prec  \alpha(z)&=&- \alpha(x)\succ(y\cdot z)=(x\succ y)\succ \alpha(z),\label{cond-Hom-anti-dend1}\\
(x\prec y)\succ \alpha(z)&=&\alpha(x)\prec(y\succ z),\label{cond-Hom-anti-dend2}
 \end{eqnarray}
  where $x\cdot y=x\prec y+x\succ y$, for all $x,y,z\in\mathcal A$.
\end{df}
\begin{exa}

Let $\mathcal{A} = \mathcal{A}_{\overline{0}}\oplus \mathcal{A}_{\overline{1}}$ be a $3$-dimensional superspace, where $\mathcal{A}_{\overline{0}}=<e_0>$and $\mathcal{A}_{\overline{1}}=<e_2,e_3>$.
We define an even bilinear map $'\circ':\mathcal{A}\times\mathcal{A}\to\mathcal{A}$ by
$$\alpha(e_0)=-e_0\;,\;\alpha(e_1)=-2e_1\;,\;\alpha(e_2)=e_2,$$
and the following nonzero products
$$e_1\prec e_1=-2e_0\;,\; e_1\prec e_0=-\frac{1}{2}e_2\;,\; e_0\prec e_1=e_2 $$
$$ e_1\succ e_1 =-2e_0\;,\; e_0\succ e_1=-\frac{1}{2}e_2\;,\; e_1\succ e_0=e_2$$

\end{exa}
\begin{rem}
If $\alpha=id$, we recover anti-dendriform superalgebras structures.    
\end{rem}
\begin{prop}\label{twist-dendr}
Let $(\mathcal A,\prec,\succ)$ be an anti-dendriform superalgebra and $\alpha:\mathcal A\to\mathcal A$ be an algebra morphism. Then, the quadruplet $(\mathcal A,\prec_\alpha,\succ_\alpha,\alpha)$, where $\prec_\alpha,\succ_\alpha$ are defined by 
 \begin{align}
x\prec_\alpha y&=\alpha(x)\prec\alpha(y), \label{twist-anti-dendriform1}\\  
x\succ_\alpha y&=\alpha(x)\succ\alpha(y), \label{twist-anti-dendriform2}
 \end{align}
 is a multiplicative Hom-anti-dendriform superalgebra.
\end{prop}
\begin{proof}
 Let $(\mathcal A,\prec,\succ)$ be an anti-dendriform superalgebra. Then, for any homogeneous elements $x,y,z\in\mathcal A$, we have
 \begin{align*}  \alpha(x)\prec_\alpha(y\prec_\alpha z)=\alpha^2(x\prec(y\prec z))=\alpha^2\big(-(x\cdot y)\prec z)\big)=\alpha^2(- x\succ(y\cdot z)\big)=\alpha^2((x\succ y)\succ z)\big),   
 \end{align*}
 which gives that,
$$\alpha(x)\prec_\alpha(y\prec_\alpha z)=-(x\cdot_\alpha y)\prec_\alpha  \alpha(z)=- \alpha(x)\succ_\alpha(y\cdot_\alpha z)=(x\succ_\alpha y)\succ_\alpha \alpha(z).$$
Then, condition \eqref{cond-Hom-anti-dend1} is satisfied. By the same way, we can show that condition \eqref{cond-Hom-anti-dend2} is satisfied.\\
It is easy to see that
$$\alpha(x\prec_\alpha y)=\alpha^2(\alpha(x)\prec\alpha(y))=\alpha(x)\prec_\alpha\alpha(y),$$
and $$\alpha(x\succ_\alpha y)=\alpha^2(\alpha(x)\succ\alpha(y))=\alpha(x)\succ_\alpha\alpha(y).$$
Then, $(\mathcal A,\prec_\alpha,\succ_\alpha,\alpha)$ is a multiplicative Hom-anti-dendriform superalgebra.
\end{proof}
\begin{thm}
Let $(\mathcal A,\prec,\succ,\alpha)$ be a Hom-anti-dendriform algebra.
Define the bilinear map $'\cdot': \mathcal A \times\mathcal A \to\mathcal A $ by 
\begin{align*}
x\cdot y = x\prec y + x\succ y 
\end{align*}
Then $(\mathcal A,\cdot,\alpha)$ is a Hom-associative superalgebra with a representation $(\mathcal A,-\mathfrak L_{\prec},-\mathfrak R_{\succ},\alpha)$ where $\mathfrak L_{\prec},\mathfrak R_{\succ}:\mathcal A \to gl(\mathcal A)$ are even linear maps defined by 
\begin{align}  
\mathfrak L_{\prec}(x)y&=x \prec y,\label{rep-com-ass-ant-dend1}\\
\mathfrak R_{\succ}(x)y&=(-1)^{|x||y|}y\succ x.\label{rep-com-ass-ant-dend2}
\end{align}
\end{thm}
\begin{proof}
Let $x,y,z\in\mathcal A$, then by using Eqs. \eqref{cond-Hom-anti-dend1} and \eqref{cond-Hom-anti-dend2} we have:
\begin{align*}
\alpha(x)\cdot(y\cdot z)&=\alpha(x)\prec(y\prec z +y\succ z )
+\alpha(x)\succ(y\prec z+y\succ z)\\
&=(x\succ y)\succ\alpha(z)+(x\prec y)\succ\alpha(z)+\alpha(x)\succ(y\prec z)+\alpha(x)\succ(y\succ z)\\
&=(x\succ y)\succ\alpha(z)+(x\prec y)\succ\alpha(z)+\alpha(x)\succ(y\prec z)+\alpha(x)\succ(y\succ z) \\
&=(x\cdot y)\succ\alpha(z)-\alpha(x)\succ(y\succ z)-\alpha(x)\prec(y\prec z)+\alpha(x)\succ(y\succ z)\\&=(x\cdot y)\succ\alpha(z)-\alpha(x)\prec(y\prec z)\\&=(x\cdot y)\succ\alpha(z)+(x\cdot y)\prec\alpha(z)\\&=(x\cdot y)\cdot\alpha(z).
\end{align*}
Then, $(\mathcal A,\cdot,\alpha)$ is a Hom-associative superalgebra. It remains to show that $(\mathcal A,-\mathfrak L_\prec,-\mathfrak R_\succ ,\alpha)$ is a representation of $\mathcal A$.\\
The conditions of Eq. \eqref{rAs1} are clearly satisfied.
For any $x,y,z\in\mathcal H(\mathcal A)$, we have:
\begin{align*}
-\mathfrak L_{\prec}(x\cdot y)\alpha(z)&=-(x\cdot y)\prec \alpha(z) \\&\stackrel{\eqref{cond-Hom-anti-dend1}}{=}\alpha(x)\prec(y\prec z) \\
&=\mathfrak L_{\prec}(\alpha(x)) \mathfrak L_{\prec}(y)(z).
 \end{align*}
 Then $$-\mathfrak L_\prec(x\cdot y)\alpha=-\mathfrak L_\prec(\alpha(x))\big(-\mathfrak L_\prec(y)\big),$$
 which implies that condition \eqref{rAs2} is satisfied.
\begin{align*}
-\mathfrak R_\succ(x\cdot y)\alpha(z)&=-(-1)^{|z|(|x|+|y|)}\alpha(z)\succ(x\cdot y)\\&= -(-1)^{|z|(|x|+|y|)}\alpha(z)\succ(x\cdot y)&\\&\stackrel{\eqref{cond-Hom-anti-dend1}}{=}(-1)^{|z|(|x|+|y|)}(z\succ x)\succ\alpha(y)\\&= -(-1)^{|x||y|}\mathfrak R_\succ(\alpha(y))\big(-\mathfrak R_\succ(x)(z)\big).  
\end{align*}
Then, condition \eqref{rAs3} is satisfied. Similarly, we can show condition \eqref{rAs4} and therefore $(\mathcal A,-\mathfrak L_\prec,-\mathfrak R_\succ,\alpha)$ is a representation of $(\mathcal A,\cdot,\alpha)$. 
\end{proof}

\begin{df}
Let $(\mathcal A,\mu,\alpha)$ be a Hom-associative  superalgebra and $(V,\mathfrak l,\mathfrak r,\beta)$ be a bimodule. An even linear map $T:V\to\mathcal A$ is called an \textbf{anti-super-$\mathcal O$-operator} on $(\mathcal A,\mu,\alpha)$ associate with $(V,\mathfrak l,\mathfrak r,\beta)$ if the following equations holds:
\begin{eqnarray}
T\circ\beta&=&\alpha\circ T,\label{cond-ant-O-oper-ncomm-Hom-ass1} \\
\mu( T(u),T(v))&=&-T\big(\mathfrak l(T(u))v+(-1)^{|u||v|}\mathfrak r(T(v))u\big),\;\forall u,v\in\mathcal H(V).\label{cond-ant-O-oper-ncomm-Hom-ass2}
\end{eqnarray}
Furthermore, $T$ is called \textbf{strong} if
\begin{equation}\label{strong-ant-O-oper-ncomm-Hom-ass}
\mathfrak l(\mu(T(u),T(v)))\beta(w)= (-1)^{|u|(|v|+|w|)}\mathfrak r(\mu(T(v),T(w)))\beta(u),\;\forall u,v,w\in\mathcal{H}(V).   
\end{equation}
In particular, a super-anti-$\mathcal O$-operator $\mathcal R$ of $(\mathcal A,\mu,\alpha)$ associated with the bimodule $(\mathcal A,\mathfrak L,\mathcal R,\alpha)$
is called an \textbf{anti-Rota–Baxter operator (of weight zero)}; that is, $\mathcal R :\mathcal A\to\mathcal A $ is an even linear map commuting with $\alpha$ and satisfying
\begin{equation}\label{ant-RB-oper-noncomm-ass}
 \mu(\mathcal R(x),\mathcal R(y))=-\mathcal R\big(\mu(\mathcal R(x),y)+\mu(x,\mathcal R(y))\big),\;\forall x,y\in\mathcal A.   
\end{equation}
An anti-Rota–Baxter operator $\mathcal R$ is called \textbf{strong} if $\mathcal R$ satisfies
\begin{equation}\label{strong-ant-RB-oper-noncomm-ass}
\mu(\mu(\mathcal R(x),\mathcal R(y)),\alpha(z))=\mu(\alpha(x),\mu(\mathcal R(y),\mathcal R(z)),\;\forall x,y,z\in\mathcal A.    
\end{equation}
\end{df} 

\begin{prop}\label{Hom-ant-dend-by-ant-O-oper}
Let $T:V\rightarrow\mathcal A$ be an anti-super-$\mathcal O$-operator on a Hom-associative superalgebra $(\mathcal A,\mu,\alpha)$ with respect to a representation $(V,\mathfrak l,\mathfrak r,\beta)$. Let us define two binary operations $\prec_T,\succ_T:V\otimes V\to V$ by:
\begin{align}
u \prec_T v&=-\mathfrak l(T(u))v,\label{H-anti-dend-from-O-oper-H-ass1} \\
u \succ_T v&=-(-1)^{|u||v|} \mathfrak r(T(v))u,\label{H-anti-dend-from-O-oper-H-ass2}  
\end{align}
for any $u,v\in\mathcal H(V)$. Then $(V,\prec_T,\succ_T,\beta)$ is a Hom-anti-dendriform superalgerbra if and only if $T$ is strong. In this case $T$ is a homomorphism of Hom-associative superalgebras from $(V,\cdot_T,\beta)$ to $(\mathcal A,\mu,\alpha)$. Furthermore, there is an induced Hom-anti-dendriform superalgebra structure on
$T(V)=\{T(u),\;u\in V\}\subseteq A$ given by
\begin{equation}\label{ant-Hom-dend-T(V)}
T(u)\prec T(v)=T(u\prec_Tv),\;\;\;T(u)\succ T(v)=T(u\succ_Tv),\;\forall u,v\in\mathcal H(V).
\end{equation}
\end{prop}
\begin{proof}
Let $u,v,w\in\mathcal H(V)$. Then, by using the fact that $T$ is an anti-super-$\mathcal O$-operator on $(\mathcal A,\mu,\alpha)$ with respect to $(V,\mathfrak l,\mathfrak r,\beta)$ we have:
\begin{align*}
\beta(u) \prec_T(v\prec_T w)&= -\mathfrak l(T(\beta(u)))(v\prec_T w) \\
&=\mathfrak l(\alpha T(u)) \mathfrak l(T(v))w \\
&=\mathfrak l(\mu(T(u),T(v)))\beta(w) \\
&=-\mathfrak l(T\big(\mathfrak l(T(u))v+(-1)^{|u||v|}
\mathfrak r(T(v))u\big))\beta(w) \\
&=-(u\prec_T v+u\succ_T v)\prec_T\beta(w)\\&=-(u\cdot_Tv)\prec_T\beta(w).
\end{align*}
By the same way, we have
\begin{align*}
(u\succ_T v) \succ_T \beta(w)&=-(-1)^{|w|(|u|+|v|)} \mathfrak r(T(\beta(w))(u\succ_T v) \\
&=(-1)^{|w|(|u|+|v|)+|u||v|}\mathfrak r(\alpha T(w))\mathfrak r(T(v))u\\&\stackrel{\eqref{rAs3}}{=} (-1)^{|u|(|v|+|w|)}\mathfrak r\Big(\mu\big(T(v),T(w)\big)\Big)\beta(u)\\&=-(-1)^{|u|(|v|+|w|)}\mathfrak r\Big(T\big(\mathfrak l(T(v))w+(-1)^{|v||w|}\mathfrak r(T(w))v\big)\Big)\beta(u)\\&=(-1)^{|u|(|v|+|w|)}\mathfrak r\Big(T\big(v\prec_T w+(-1)^{|v||w|} w\succ_T v\big)\Big)\beta(u)\\&=-\beta(u
)\succ_T(v\prec_T w+(-1)^{|v||w|} w\succ_T v)\\&=-\beta(u)\succ_T(v\cdot_Tw).
\end{align*}
Then condition \eqref{cond-Hom-anti-dend1} is satisfied if and only if 
$$(u\cdot_Tv)\prec_T\beta(w)=\beta(u)\succ_T(v\cdot_Tw),$$
which gives that $$\mathfrak l(\mu(T(u),T(v)))\beta(w)=(-1)^{|u|(|v|+|w|)}\mathfrak r\big(\mu(T(v),T(w))\big)\beta(u),\;\forall u,v\in\mathcal H(V),$$
that is, $T$ is strong.

Similarly we can show that condition \eqref{cond-Hom-anti-dend2} is satisfied on $V$ which gives the Proof.
\end{proof}
\begin{cor}
Let $\mathcal R$ be a strong anti-Rota-Baxter operator of weight zero on an associative superalgebra  $(\mathcal A,\mu,\alpha)$. Then, $(\mathcal A,\prec_\mathcal R,\succ_\mathcal R,\alpha)$ where $\prec_\mathcal R,\succ_\mathcal R:\mathcal A\to\mathcal A$ are defined by 
\begin{align}
 x\prec_\mathcal R y&=-\mu(\mathcal R(x),y),\label{H-anti-dend-From-ARB-Hass1} \\
 x\succ_\mathcal R y&=-\mu(x,\mathcal R(y)),\label{H-anti-dend-From-ARB-Hass2}
\end{align}
for any $x,y\in\mathcal H(\mathcal A)$ is a Hom-anti-dendriform superalgebra.
\end{cor}
\begin{lem} 
An invertible anti-super-$\mathcal O$-operator of a Hom-associative superalgebra is automatically
strong.    
\end{lem}

\subsection{Hom-anti-pre-Lie superalgebras}

\begin{df}(\cite{Zhao-Liu-Chen})
An \textbf{anti-pre-Lie superalgebra} (also called \textbf{anti-left-symmetric superalgebra}) is a $\mathbb Z_2$-graded vector space $\mathcal{A}$ equipped with an even bilinear map $\circ:\mathcal A\times\mathcal A\to\mathcal A$ such that, for all $x,y,z\in\mathcal H(\mathcal A)$, we have the following conditions   \begin{eqnarray}
 &&x\circ(y\circ z)-(-1)^{|x||y|} y\circ(x\circ z)=(-1)^{|x||y|}[y,x]_\circ \circ z,\label{cond-ant-pre-Lie1}  \\
 &&(-1)^{|x||z|}[x,y]_\circ\circ z+(-1)^{|x||y|}[y,z]_\circ\circ x+(-1)^{|y||z|}[z,x]_\circ\circ y=0\label{cond-ant-pre-Lie2},
\end{eqnarray} 
where 
\begin{equation}\label{comp-Lie-struc}
[x,y]_\circ=x\circ y-(-1)^{|x||y|}y\circ x.
\end{equation}

\end{df}
\begin{rem}\
  Condition \eqref{cond-ant-pre-Lie1} is equivalent to 
\begin{equation}\label{exp-equiv-anti-pre-Lie1}\mathfrak{aas}_\circ(x,y,z)=(-1)^{|x||y|}\mathfrak{aas}_\circ(y,x,z),\;\forall x,y,z\in\mathcal{H}(\mathcal{A}).
\end{equation} 
\end{rem}
\begin{exa}
Let $\mathcal{A} = \mathcal{A}_{\overline{0}}\oplus \mathcal{A}_{\overline{1}}$ be a $3$-dimensional superspace, where $\mathcal{A}_{\overline{0}}=<e_1>$and $\mathcal{A}_{\overline{1}}=<e_2,e_3>$.
We define an even bilinear map $\circ:\mathcal{A}\times\mathcal{A}\to\mathcal{A}$ by
$$e_1\circ e_2=-e_2,\;e_1\circ e_3=e_3,\;e_2\circ e_1=-2e_2$$ 
$$e_3\circ e_1=2e_3,\;e_2\circ e_3=-e_3\circ e_2=\frac{1}{4}e_1.$$
Then, $(\mathcal{A},\circ)$ is an anti-pre-Lie superalgebra
\end{exa}
\begin{df}
A \textbf{Hom-anti-pre-Lie superalgebra} (also called \textbf{Hom-anti-left-symmetric superalgebra}) is a triplet $(\mathcal A,\circ,\alpha)$ consisting of a $\mathbb Z_2$-graded vector space $\mathcal{A}$, an even bilinear map $\circ:\mathcal A\times\mathcal A\to\mathcal A$ and an even linear map $\alpha:\mathcal A\to\mathcal A$ such that, the following conditions hold   \begin{eqnarray}
 &&\alpha(x)\circ(y\circ z)-(-1)^{|x||y|} \alpha(y)\circ(x\circ z)=[y,x]_\circ \circ \alpha(z),\label{cond-Hom-ant-pre-Lie1}  \\
 &&(-1)^{|x||z|}[x,y]_\circ\circ \alpha(z)+(-1)^{|x||y|}[y,z]_\circ\circ \alpha(x)+(-1)^{|y||z|}[z,x]_\circ\circ \alpha(y)=0\label{cond-Hom-ant-pre-Lie2},
\end{eqnarray} 
for any $x,y,z\in\mathcal{H}(\mathcal{A})$, where $[x,y]_\circ=x\circ y-(-1)^{|x||y|}y\circ x$.\\
A Hom-anti-pre-Lie superalgebra $(\mathcal A,\circ,\mu)$ is called \textbf{regular} if $\alpha$ is bijective, and called \textbf{multiplicative} if $\alpha$ is an algebra endomorphism, i.e., $\alpha(x\circ y)=\alpha(x)\circ\alpha(y)$ for any $x,y\in\mathcal A$.

\end{df}
\begin{rem}\
 \begin{enumerate}
\item If $\alpha=id$, we recover \textbf{anti-pre-Lie superalgebras} structures (see \cite{Zhao-Liu-Chen} for more details).
\item Condition \eqref{cond-Hom-ant-pre-Lie1} is equivalent to
\begin{equation}\label{cond-equiv-Hom-anti-pre-Lie1}
\mathfrak{aas}_\circ^\alpha(x,y,z)=(-1)^{|x||y|}\mathfrak{aas}_\circ^\alpha(y,x,z),\;\forall x,y,z\in\mathcal{H}(\mathcal A). 
\end{equation}
 \end{enumerate}   
\end{rem}
\begin{exa}
Let $\mathcal A=\mathcal A_{\overline 0}\oplus \mathcal A_{\overline 1}$ be a
$\mathbb Z_2$-graded vector space defined by
\[
\mathcal A_{\overline 0}=\langle e_1,e_2\rangle,
\qquad
\mathcal A_{\overline 1}=\langle e_3\rangle .
\]
Define an even bilinear map $
\circ:\mathcal A\otimes \mathcal A \longrightarrow \mathcal A
$
by
\[
e_1\circ e_2 = e_3,\qquad
e_2\circ e_1 = -\,e_3,
\]
and all other products are zero.
Let $\alpha:\mathcal A\to\mathcal A$ be the even linear map defined by
\[
\alpha(e_1)=e_1,\qquad
\alpha(e_2)=2e_2,\qquad
\alpha(e_3)=2e_3.
\]
Then  the triplet $(\mathcal A,\circ,\alpha)$
is a $3$-dimensional Hom-anti-pre-Lie superalgebra.
\end{exa}
Recall that $(\mathcal A,\circ,\alpha)$ is called \textbf{Hom-Lie-admissible superalgebra} where $\mathcal A$ is vector superspace with an even bilinear map $\mu:\mathcal A\times\mathcal A\to\mathcal A$ and an even linear map $\alpha:\mathcal A\to\mathcal A$, if the product $[\cdot,\cdot]:\mathcal A\times\mathcal A\to\mathcal A$ defined by Eq. \eqref{comp-Lie-struc} makes $(\mathcal A,[\cdot,\cdot],\alpha)$ a Hom-Lie superalgebra (see \cite{Faouzi-Abdenacer} for more details). We call $(\mathcal A,[\cdot,\cdot],\alpha)$ in this case the \textbf{sub-adjacent Hom-Lie superalgebra} of  $(\mathcal A,\circ,\alpha)$ denoted by $\mathfrak g(\mathcal A)$ and $(\mathcal A,\circ,\alpha)$ is called a \textbf{compatible} (Hom-Lie admissible) superalgebra structure on the Hom-Lie superalgebra $\mathfrak g(\mathcal A)$.
\begin{prop}
Let $\mathcal A$ be a vector superspace with an even bilinear map $\circ:\mathcal A\times\mathcal A\to\mathcal A$ and an even linear map $\alpha:\mathcal A\to\mathcal A$. Then the
following assertions are equivalent: 
\begin{enumerate}
\item $(\mathcal A,\circ,\alpha)$ is a Hom-anti-pre-Lie superalgebra.
\item For $(\mathcal A,\circ,\alpha)$, Eq. \eqref{cond-Hom-ant-pre-Lie1} and for any $x,y,z\in\mathcal H(\mathcal A)$, the following equation hold
\begin{equation}\label{cond-equiv1}
\circlearrowleft_{x,y,z}(-1)^{|x||z|}\alpha(x)\circ[y,z]_\circ=(-1)^{|x||z|}\alpha(x)\circ[y,z]_\circ+ (-1)^{|x||y|}\alpha(y)\circ[z,x]_\circ+(-1)^{|y||z|}\alpha(z)\circ[x,y]_\circ=0.
\end{equation}
\item $(\mathcal A,\circ,\alpha)$ is a Hom-Lie admissible superalgebra, such that $(\mathcal A,-\mathfrak L_\circ,\alpha)$ is a representation of the sub-adjacent Hom-Lie superalgebra $(\mathfrak g(\mathcal A),[\cdot,\cdot]_\circ,\alpha)$, where $\mathfrak L_\circ:\mathfrak g(\mathcal A)\to End(\mathcal A)$ is an even linear map defined by $\mathfrak L_\circ(x)(y)=x\circ y,\;\forall x,y\in\mathcal H(\mathcal A)$.
\end{enumerate}
\end{prop}
\begin{proof}
$(1)\Longleftrightarrow(2)$. Let $x,y,z\in\mathcal{H}(\mathcal{A})$. Suppose that Eq. \eqref{cond-Hom-ant-pre-Lie1} holds. Then, by using Eqs. \eqref{comp-Lie-struc} and \eqref{cond-Hom-ant-pre-Lie1}, we have:
\begin{align}
\circlearrowleft_{x,y,z}(-1)^{|x||z|}\alpha(x)\circ[y,z]_\circ&=(-1)^{|x||z|}\alpha(x)\circ[y,z]_\circ+ (-1)^{|x||y|}\alpha(y)\circ[z,x]_\circ+(-1)^{|y||z|}\alpha(z)\circ[x,y]_\circ\nonumber\\&=(-1)^{|x||z|}\alpha(x)\circ\big(y\circ z-(-1)^{|y||z|}z\circ y\big)\nonumber\\&+ (-1)^{|x||y|} \alpha(y)\circ\big(z\circ x-(-1)^{|x||z|}x\circ z\big)\nonumber\\&+ (-1)^{|y||z|} \alpha(z)\circ\big(x\circ y-(-1)^{|x||y|}y\circ x)\nonumber\\&=(-1)^{|x||z|}\big(\alpha(x)\circ(y\circ z)-(-1)^{|x||y|}\alpha(y)\circ(x\circ z)\big)\nonumber\\&+(-1)^{|x||y|}\big(\alpha(y)\circ(z\circ x)-(-1)^{|y||z|}\alpha(z)\circ(y\circ x)\big)\nonumber \\&+(-1)^{|y||z|}\big(\alpha(z)\circ(x\circ y)-(-1)^{|x||z|}\alpha(x)\circ(z\circ y)\big)\nonumber\\&=(-1)^{|x||z|}[y,x]_\circ\circ\alpha(z)+(-1)^{|x||y|}[z,y]_\circ\circ\alpha(x)\label{cond-equiv0}\\&+(-1)^{|y||z|}[x,z]_\circ\circ\alpha(y),\nonumber
\end{align}
which gives that Eq. \eqref{cond-Hom-ant-pre-Lie2} holds if and only if Eq. \eqref{cond-equiv1} is also.\\

$(1)\Longleftrightarrow(3)$. Suppose that $(\mathcal A,\circ,\alpha)$ is a Hom-anti-pre-Lie superalgebra. Then, by Eqs. \eqref{cond-ant-pre-Lie2} and \eqref{cond-equiv1}, for any homogeneous elements $x,y,z$ of $\mathcal A$, we have:
\begin{align*}
\circlearrowleft_{x,y,z}(-1)^{|x||z|}[\alpha(x),[y,z]_\circ]_\circ&=(-1)^{|x||z|}[\alpha(x),[y,z]_\circ]_\circ+(-1)^{|x||y|}[\alpha(y),[z,x]_\circ]_\circ\\&+(-1)^{|y||z|}[\alpha(z),[x,y]_\circ]_\circ  \\&=(-1)^{|x||z|}\alpha(x)\circ[y,z]_\circ+(-1)^{|x||y|}\alpha(y)\circ[z,x]_\circ+(-1)^{|y||z|}\alpha(z)\circ[x,y]_\circ  \\&- \big((-1)^{|x||y|}[y,z]_\circ\circ\alpha(x)+(-1)^{|y||z|}[z,x]_\circ\circ\alpha(y)+(-1)^{|x||z|}[x,y]_\circ\circ\alpha(z) \big)\\&=0.
\end{align*}
Thus $(\mathcal A,[\cdot,\cdot]_\circ,\alpha)$ is a Hom-Lie superalgebra and $(\mathcal A,\circ,\alpha)$ is Hom-Lie admissible superalgebra. In addition, Eq. \eqref{cond-Hom-ant-pre-Lie1} gives the following relation:
\begin{equation}\label{equiv-res-repr}
\mathfrak L_\circ(x)\mathfrak L_\circ(y)-(-1)^{|x||y|}\mathfrak L_\circ(y)\mathfrak L_\circ(x)=-\mathfrak L_\circ([x,y]),\;\forall x,y\in\mathcal H(\mathcal A).    
\end{equation}
Then $(\mathcal A,-\mathfrak L_\circ,\alpha)$ is a representation of the Hom-Lie superalgebra $(\mathcal A,[\cdot,\cdot]_\circ,\alpha)$.\\
Now, suppose that $(\mathcal A,\circ,\alpha)$ is a Hom-Lie-admissible superalgebra such that $(\mathcal A,-\mathfrak L_\circ,\alpha)$ is a representation of the Hom-Lie superalgebra $(\mathfrak g(\mathcal A),[\cdot,\cdot]_\circ,\alpha)$. Then, we can notice that Eq. \eqref{equiv-res-repr} holds which is equivalent to the fact that Eq. \eqref{cond-Hom-ant-pre-Lie1} is also holds. Moreover, Eq. \eqref{cond-equiv0} holds, too. In addition, by using the fact that $(\mathcal A,\circ,\alpha)$ is a Hom-Lie admissible superalgebra, we have
$$\circlearrowleft_{x,y,z}(-1)^{|x||z|}\alpha(x)\circ[y,z]_\circ=\circlearrowleft_{x,y,z}(-1)^{|x||z|}[x,y]_\circ\circ\alpha(z),\;\forall x,y,z\in\mathcal{H}(\mathcal{A}).$$
Thus by Eq. \eqref{cond-equiv0}, Eqs. \eqref{cond-Hom-ant-pre-Lie2} and \eqref{cond-equiv1} hold which gives that $(\mathcal{A},\circ,\alpha)$ is a Hom-anti-pre-Lie superalgebra.

\end{proof}
\begin{prop}\label{prop-twist-anti-pre-Lie}
Let $(\mathcal A,\circ)$ be an anti-pre-Lie superalgebra and $\alpha$ is an algebra endomorphism of $\mathcal A$. Then, the triplet $(\mathcal A,\circ_\alpha,\alpha)$ is a multiplicative Hom-anti-pre-Lie superalgebra, where
\begin{equation}\label{twist-anti-pre-Lie}
x\circ_\alpha y=\alpha(x)\circ\alpha(y),\;\forall x,y\in\mathcal A.    
\end{equation}
\end{prop}
\begin{proof}
 It's obvious to show that, $\alpha([x,y]^C)=[\alpha(x),\alpha(y)]^C$ and $\alpha(x\circ_\alpha y)=\alpha(x)\circ_\alpha \alpha(y)$, for all $x,y\in\mathcal H(\mathcal A)$. Let $x,y,z\in\mathcal H(\mathcal A)$, then by using Eq. \eqref{cond-ant-pre-Lie1} and the fact that $\alpha$ is an endomorphism algebra, we have
 \begin{align*}
 \alpha(x)\circ_\alpha(y\circ_\alpha z)-(-1)^{|x||y|} \alpha(y)\circ_\alpha(x\circ_\alpha z)&=\alpha^2(x\circ(y\circ z)-(-1)^{|x||y|} y\circ(x\circ z))\\&=\alpha^2([y,x]_\circ \circ z)\\&=\alpha^2([y,x]_\circ)\circ\alpha^2(z)\\&= \alpha([y,x]_\circ)\circ_\alpha\alpha(z)\\&=\alpha(x\circ y-(-1)^{|x||y|}y\circ x)\circ_\alpha\alpha(z)\\&= (x\circ_\alpha y-(-1)^{|x||y|}y\circ_\alpha x)\circ_\alpha\alpha(z)\\&= [x,y]_{\circ_\alpha}\circ_\alpha\alpha(z),  
 \end{align*}
which implies that condition \eqref{cond-Hom-ant-pre-Lie1} is satisfied. By the same way, we can show that condition \eqref{cond-Hom-ant-pre-Lie2} is satisfied. Then $(\mathcal{A},\circ_\alpha,\alpha)$ is a multiplicative Hom-anti-pre-Lie superalgebra. 
\end{proof}

\subsection{Anti-super-$\mathcal{O}$-Operators on Hom-Lie superalgebras}\ \\

Let us consider a Hom-Lie superalgebra $(\mathfrak g,[\cdot,\cdot],\alpha)$ and a representation $(V,\rho,\beta)$ of $\mathfrak g$. Recall that, an even linear map $T:V\to\mathfrak g$ is called a \textbf{super-$\mathcal O$-operator} on $(\mathfrak g,[\cdot,\cdot],\alpha)$ associated to $(V,\rho,\beta)$ if $T$ satifyies
\begin{eqnarray}
 T\circ\beta&=&\alpha\circ T,\label{cond-O-Oper-Hom-Lie1}\\
\lbrack T(u),T(v)\rbrack&=&T\big(\rho(T(u))v-(-1)^{|u||v|}\rho(T(v))u\big),\;\forall u,v\in\mathcal H(V).\label{cond-O-Oper-Hom-Lie2}
\end{eqnarray}

\begin{df}
Let $(\rho,V,\beta)$ be a representation of a Hom-Lie superalgebra $(\mathfrak g,[\cdot,\cdot],\alpha)$. An even linear map $T:V\to\mathfrak g$ is called an \textbf{anti-super-$\mathcal{O}$-operator} on $(\mathfrak g,[\cdot,\cdot],\alpha)$ associated to $(V,\rho,\beta)$ if it satisfies
\begin{eqnarray}
T\circ\beta&=&\alpha\circ T,\label{cond-ant-O-oper-Hom-Lie1} \\
\lbrack T(u),T(v)\rbrack&=&T\big((-1)^{|u||v|}\rho(T(v))u-\rho(T(u))v\big),\;\forall u,v\in\mathcal H(V).\label{cond-ant-O-oper-Hom-Lie2}
\end{eqnarray}
an anti-super-$\mathcal O$-operator $T$ is called \textbf{strong} if it satisfies:
\begin{equation}\label{cond-strong-O-Oper}
\rho([T(u),T(v)])\beta(w)+(-1)^{|u|(|v|+|w|)}\rho([T(v),T(w)])\beta(u)+(-1)^{|w|(|u|+|v|)}\rho([T(w),T(u)])\beta(v)=0,  
\end{equation}
for all $ u,v,w\in\mathcal H(V)$.\\
In particular, an anti-super-$\mathcal O$-operator $\mathcal R$ of $(\mathfrak g,[\cdot,\cdot],\alpha)$ associated to the adjoint representation $(\mathfrak g,ad,\alpha)$ is called an \textbf{anti-Rota-Baxter operator (of weight zero)}, that is, $\mathcal R:\mathfrak g\to\mathfrak g$ is an even linear map commuting with $\alpha$ satisfying:
\begin{equation}\label{cond-Rota-Baxter-Oper}
 [\mathcal R(y),\mathcal R(x)]= (-1)^{|x||y|}\mathcal R([\mathcal R(x),y]+[x,\mathcal R(y)]),\;\forall x,y\in\mathfrak g.  
\end{equation}
An anti-Rota-Baxter operator $\mathcal R$ is called \textbf{strong}, if it satisfies
\begin{equation}\label{cond-strong-Rota-Baxter-Oper}
[[\mathcal R(x),\mathcal R(y)],\alpha(z)]+(-1)^{|x|(|y|+|z|)} [[\mathcal R(y),\mathcal R(z)],\alpha(x)]+(-1)^{|z|(|x|+|y|)}[[\mathcal R(z),\mathcal R(x)],\alpha(y)]=0,   
\end{equation}
for all $x,y,z\in\mathcal H(\mathfrak g)$.
\end{df}
\begin{df}
Let $(V,\rho,\beta)$ be a representation of a Hom-Lie superalgebra $(\mathfrak g,[\cdot,\cdot],\alpha)$. A linear map $\mathfrak D:\mathfrak g\to V$ is called a \textbf{super anti-derivation} ( or a \textbf{super anti-$1$-cocycle}) if $\mathfrak D$ satisfies
\begin{eqnarray}
\mathfrak D\circ\alpha&=&\beta\circ\mathfrak D,\label{anti-der-Hom-Lie1}\\
\mathfrak D([x,y])&=&(-1)^{|y|(|\mathfrak D|+|x|)}\rho(y)(\mathfrak D(x))-\rho(x)(\mathfrak D(y)),\;\forall x,y\in\mathfrak g.\label{anti-der-Hom-Lie2}    
\end{eqnarray}
\end{df}
\begin{prop}
Let $(V,\rho,\beta)$ be a representation of a Hom-Lie superalgebra $(\mathfrak g,[\cdot,\cdot],\alpha)$ and $T:V\to\mathfrak g$ an even invertible linear map. Then $T$ is an anti-super-$\mathcal O$-operator on $\mathfrak g$ associated to $(V,\rho,\beta)$ if and only if $T^{-1}$ is an even anti-derivation on $\mathfrak g$.
\end{prop}
\begin{proof}
Let $T:V\to\mathfrak g$ be an even invertible linear map. \\
Suppose that $T$ is an anti-super-$\mathcal O$-operator on $(\mathfrak g,[\cdot,\cdot],\alpha)$ with respect to a representation $(V,\rho,\beta)$. Then, by Eq. \eqref{anti-der-Hom-Lie1}, we have 
$$ T^{-1}\circ T\circ\alpha\circ T^{-1}=T^{-1}\circ\beta\circ T\circ T^{-1},$$
which gives that 
$$\alpha\circ T^{-1}=T^{-1}\circ\beta.$$
So, $T^{-1}$ satisfies Eq. \eqref{anti-der-Hom-Lie1}.\\
Let $x,y\in\mathcal{H}(\mathfrak g)$. By using Eq. \eqref{cond-ant-O-oper-Hom-Lie2}, we have
\begin{align*}
T^{-1} [x,y]&=T^{-1} [T(T^{-1}(x)),T(T^{-1}(y))]\\&=T^{-1}\Big(T\big((-1)^{|x||y|}\rho(T(T^{-1}(y)))(T^{-1}(x))-\rho(T(T^{-1}(x)))(T^{-1}(y))\big)\Big)\\&=(-1)^{|x||y|}\rho(y)(T^{-1}(x))-\rho(x)(T^{-1}(y)),   
\end{align*}
which implies that, $T^{-1}$ satisfies Eq. \eqref{anti-der-Hom-Lie2}. Then, $T^{-1}$ is an even anti-derivation on $(\mathfrak g,[\cdot,\cdot],\alpha)$.\\
Similar we can show the converse case.
\end{proof}
In general, when \( T \) is a super-\( \mathcal{O} \)-operator on a Hom-Lie superalgebra \( (\mathfrak{g}, [\cdot,\cdot], \alpha) \) associated with a representation \( (V, \rho, \beta) \), one can define an induced Hom-pre-Lie superalgebra structure on \( V \) (see \cite{Mabrouk-Ncib-Silvestrov,Mabrouk-Silvestrov-Zouaidi} for further details). However, this result does not hold in our setting. Specifically, an anti-super-\( \mathcal{O} \)-operator \( T \) on a Hom-Lie superalgebra with respect to a representation \( (V, \rho, \beta) \) does not, in general, give rise to an induced Hom-anti-pre-Lie superalgebra structure on \( V \), except in the case where \( T \) satisfies the stronger condition outlined in the following theorem:

\begin{thm}\label{Hom-ant-pre-Lie-by-ant-O-oper}
 Let $T$ be an anti-super-$\mathcal O$-operator on a Hom-Lie superalgebra $(\mathfrak g,[\cdot,\cdot],\alpha)$ with respect to a representation $(V,\rho,\beta)$. Then, $(V,\circ_V,\beta)$ where, $\circ_V:V\times V\to V$ is defined by
 \begin{equation}\label{ant-O-oper-to-anti-pre-Lie}
 u\circ_V v=-\rho(T(u))v,\;\forall u,v\in V,    
 \end{equation}
 satisfies Eq. \eqref{cond-Hom-ant-pre-Lie1}. Moreover, $(V,\circ_V,\beta)$ is a Hom-Lie-admissible such that $(V,\circ_V,\beta)$ is a Hom-anti-pre-Lie superalgebra if and only if $T$ is strong. So, $T$ is a homomorphism of Hom-Lie superalgebras from the sub-adjacent Hom-Lie superalgebra $(\mathfrak g(V),[\cdot,\cdot]_{\circ_V},\beta)$ to $(\mathfrak g,[\cdot,\cdot],\alpha)$. Therefore, there is an induced Hom-anti-pre-Lie
superalgebra structure on $T(V)=\{T(u);\;u\in V\}\subset\mathfrak g$ given by
\begin{equation}\label{ind-Hom-anti-T(V)}
T(u)\circ_\mathfrak gT(v)=T(u\circ_Vv),\forall u,v\in V,
\end{equation}
and $T$ is a homomorphism of Hom-anti-pre-Lie superalgebras.
\end{thm}
\begin{proof}
For any $u,v,w\in\mathcal H(V)$, we have:
\begin{align*}
\beta(u)\circ_V(v\circ_V w)-(-1)^{|u||v|}\beta(v)\circ_V(u\circ_V w)&=\rho(T(\beta(u))\rho(T(v))w-(-1)^{|u||v|} \rho(T(\beta(v))\rho(T(u))w\\
&=\rho(\lbrack T(u),T(v)\rbrack)\beta(w)\\
&=\rho\big(T((-1)^{|u||v|}\rho(T(v))u-\rho(T(u))v)\big)\beta(w) \\
&=\big(\rho(T(u))v)-(-1)^{|u||v|}\rho(T(v))u\big)\circ_V\beta(w) \\
&=-(u\circ_V v-(-1)^{|u||v|}v\circ_V u)\circ_V\beta(w) \\
&=(-1)^{|u||v|}\lbrack v,u\rbrack_{\circ_V}\circ_V\beta(w).
\end{align*} 
Then, the product $\circ_V$ satisfies condition \eqref{cond-Hom-ant-pre-Lie1} on $V$. It remain to show that condition \eqref{cond-Hom-ant-pre-Lie2} is satisfied. Let $u,v,w\in\mathcal H(V)$, we have:
\begin{align*}
    &(-1)^{|u||w|}[u,v]_{\circ_V}\circ_V \beta(w)+(-1)^{|u||v|}[v,w]_{\circ_V}\circ_V \beta(u)+(-1)^{|v||w|}[w,u]_{\circ_V}\circ_V \beta(v)\\
&= -(-1)^{|u||w|}\rho(T([u,v]_{\circ_V}))\beta(w)
   -(-1)^{|u||v|}\rho(T([v,w]_{\circ_V}))\beta(u)
   -(-1)^{|v||w|}\rho(T([w,u]_{\circ_V}))\beta(v)\\
&= -(-1)^{|u||w|}\rho([T(u),T(v)])\beta(w)
   -(-1)^{|u||v|}\rho([T(v),T(w)])\beta(u)
   -(-1)^{|v||w|}\rho([T(w),T(u)])\beta(v)\\
&= -\Big(\rho([T(u),T(v)])\beta(w)
   +(-1)^{|u|(|v|+|w|)}\rho([T(v),T(w)])\beta(u)
   +(-1)^{|w|(|u|+|v|)}\rho([T(w),T(u)])\beta(v)\Big)\\
&=0,
\end{align*}
if and only if $T$ is strong which gives the result.
\end{proof}
\begin{cor}
Let $\mathcal R:\mathfrak g\to\mathfrak g$ be a strong anti-Rota-Baxter operator of a Hom-Lie superalgebra $(\mathfrak g,[\cdot,\cdot],\alpha)$. Then, the product 
\begin{equation}\label{ant-R-B-to-anti-pre-Lie}
x\circ y=-[\mathcal R(x),y],\;\forall x,y\in\mathfrak g,    
\end{equation}
defines on $\mathfrak g$ a Hom-anti-pre-Lie superalgebra structure. Conversely, if $\mathcal R:\mathfrak g\to\mathfrak g$ is a linear transformation of a Hom-Lie superalgebra $(\mathfrak g,[\cdot,\cdot],\alpha)$ such that, the product defined by Eq. \eqref{ant-R-B-to-anti-pre-Lie} defines on $\mathcal A$ a Hom-anti-pre-Lie superalgebra, then $\mathcal R$ satisfies Eq. \eqref{cond-strong-Rota-Baxter-Oper} and the following equation:
\begin{equation}\label{Cor-R-B-Hom-anti-pre-Lie}
[[\mathcal R(x),\mathcal R(y)]+\mathcal R([\mathcal R(x),y]+[x,\mathcal R(y)]),\alpha(z)]=0,\forall x,y,z\in\mathfrak g.
\end{equation}
\end{cor}
\begin{proof}
By theorem \ref{Hom-ant-pre-Lie-by-ant-O-oper}, the binary operation defined by Eq. \eqref{Cor-R-B-Hom-anti-pre-Lie} defines on $\mathcal A$ a Hom-anti-pre-Lie superalgebra.

Conversely, let $x,y,z$ be homogeneous elements of $\mathfrak g$. Then we have
\begin{align*}
\alpha(x)\circ(y\circ z)-(-1)^{|x||y|}\alpha(y)\circ(x\circ z)
&=(-1)^{|x||y|}[y,x]_\circ\circ \alpha(z),\\
\alpha(x)\circ[\mathcal R(y),z]-(-1)^{|x||y|}\alpha(y)\circ[\mathcal R(x),z]
&=(-1)^{|x||y|}[\mathcal R([y,x]_\circ),\alpha(z)],\\
[\alpha \mathcal R(x),[\mathcal R(y),z]]-(-1)^{|x||y|}[\alpha \mathcal R(y),[\mathcal R(x),z]]
&=(-1)^{|x||y|}[\mathcal R([\mathcal R(y),x]+[y,\mathcal R(x)]),\alpha(z)].
\end{align*}

By the Hom-super-Jacobi identity, it follows that
\begin{align*}
[[\mathcal R(x),\mathcal R(y)],\alpha(z)]
= (-1)^{|x||y|}[\mathcal R([\mathcal R(y),x]+[y,\mathcal R(x)]),\alpha(z)].
\end{align*}
Consequently,
\[
[[\mathcal R(x),\mathcal R(y)]+\mathcal R([\mathcal R(x),y]+[x,\mathcal R(y)]),\alpha(z)]=0,
\]
which implies that $\mathcal R$ satisfies condition \eqref{Cor-R-B-Hom-anti-pre-Lie}.

Assume that the product defined by Eq.~\eqref{ant-R-B-to-anti-pre-Lie} endows $\mathcal A$ with a Hom-anti-pre-Lie superalgebra structure. Then, for any $x,y,z\in\mathcal H(\mathcal A)$, we have
\begin{align*}
0=&[x,y]_\circ\circ \alpha(z)
+(-1)^{|x|(|y|+|z|)}[y,z]_\circ\circ \alpha(x)
+(-1)^{|z|(|x|+|y|)}[z,x]_\circ\circ \alpha(y)\\
=&[\mathcal R([x,y]_\circ),\alpha(z)]
+(-1)^{|x|(|y|+|z|)}[\mathcal R([y,z]_\circ),\alpha(x)]
+(-1)^{|z|(|x|+|y|)}[\mathcal R([z,x]_\circ),\alpha(y)]\\
=&[\mathcal R([\mathcal R(x),y]+[x,\mathcal R(y)]),\alpha(z)]
+(-1)^{|x|(|y|+|z|)}[\mathcal R([\mathcal R(y),z]+[y,\mathcal R(z)]),\alpha(x)]\\
&+(-1)^{|z|(|x|+|y|)}[\mathcal R([\mathcal R(z),x]+[z,\mathcal R(x)]),\alpha(y)].
\end{align*}
Hence,
\[
[[\mathcal R(x),\mathcal R(y)],\alpha(z)]
+(-1)^{|x|(|y|+|z|)}[[\mathcal R(y),\mathcal R(z)],\alpha(x)]
+(-1)^{|z|(|x|+|y|)}[[\mathcal R(z),\mathcal R(x)],\alpha(y)]=0,
\]
which shows that $\mathcal R$ satisfies condition \eqref{cond-strong-Rota-Baxter-Oper}.

\end{proof}

\section{Noncommutative Hom-pre-Poisson superalgebras}\label{Sec4}


In this section, we introduce the notions of noncommutative Hom-anti-pre-Poisson superalgebras. We explore the structural relationship between noncommutative Hom-Poisson superalgebras and noncommutative Hom-anti-pre-Poisson superalgebras through the framework of strong anti-super-$\mathcal O$-operators. Furthermore, we establish several related results and provide illustrative examples.
\begin{df}
A \textbf{noncommutative Hom-pre-Poisson superalgebra} is a quintuple $(A,\circ,\prec,\succ,\alpha)$ such that $(A,\circ,\alpha)$ is a Hom-pre-Lie superalgebra and $(A,\prec,\succ,\alpha)$ is a Hom-dendriform superalgebra satisfying the following compatibility conditions :
    \begin{align}
(x\circ y-(-1)^{|x||y|}y \circ x)\succ \alpha(z)&=\alpha(x)\circ(y\succ z)-(-1)^{|x||y|}\alpha(y)\succ(x \circ z),\label{cond-noncomm-Hom-pre-Poisson1}\\
\alpha(x)\prec(y\circ z -(-1)^{|y||z|} z \circ y)&=(-1)^{|x||y|} \alpha(y) \circ (x\prec z)-(-1)^{|x||y|}(y\circ x)\prec \alpha(z), \label{cond-noncomm-Hom-pre-Poisson2}\\
(x\succ y + x\prec y)\circ\alpha(z)&=(-1)^{|y||z|} (x\circ z)\prec \alpha(y)+\alpha(x) \succ(y\circ z).\label{cond-noncomm-Hom-pre-Poisson3}
\end{align}
\end{df}
\begin{rem}
When $\alpha=id$, we recover the noncommutative pre-Poisson superalgebras structures.    
\end{rem}
\begin{exa}\label{ex-noncom-H-pre-Pois}
 Let $\mathcal A=\mathfrak{osp}(1|2)
=\mathcal A_{\overline 0}\oplus\mathcal A_{\overline 1},
$
where $\mathcal A_{\overline 0}=\langle H,E,F\rangle$ and $\mathcal A_{\overline 1}=\langle X,Y\rangle
$.
Define a linear map
$\alpha:\mathcal A\to\mathcal A$ by
\[
\alpha(H)=H,\quad
\alpha(E)=\frac{1}{2} E,\quad
\alpha(F)=2F,\quad
\alpha(X)= X,\quad
\alpha(Y)=Y.
\]

\medskip

\noindent

Define a bilinear map $\circ:\mathcal A\times\mathcal A\to\mathcal A$
by the following nonzero products:
\[
\begin{aligned}
&H\circ E=2E,\qquad H\circ F=-2F,\qquad E\circ F=H,\\
&H\circ X=X,\qquad H\circ Y=-Y,\qquad E\circ Y=X,\qquad F\circ X=Y,\\
&X\circ X=2E,\qquad Y\circ Y=-2F,\qquad X\circ Y=H.
\end{aligned}
\]
and two bilinear maps $\prec,\succ:\mathcal A\times\mathcal A\to\mathcal A$ by:

\medskip
\noindent

\[
\begin{aligned}
&H\prec E=2E,\qquad E\prec F=H,\\
&H\prec X=X,\qquad E\prec Y=X,\\
&X\prec Y=H,\qquad X\prec X=2E.
\end{aligned}
\]

\medskip
\noindent
\[
\begin{aligned}
&E\succ H=-2E,\qquad F\succ E=-H,\\
&X\succ H=-X,\qquad Y\succ E=-X,\\
&Y\succ X=-H,\qquad Y\succ Y=2F.
\end{aligned}
\]
All other products are zero.
\medskip

\noindent
Define
\[
x\circ_\alpha y=\alpha(x\circ y),\qquad
x\prec_\alpha y=\alpha(x\prec y),\qquad
x\succ_\alpha y=\alpha(x\succ y),
\quad \forall x,y\in\mathcal A.
\]

\medskip

Then $(\mathfrak{osp}(1|2),\circ_\alpha,\prec_\alpha,\succ_\alpha,\alpha)$ is a non-commutative Hom-pre-Poisson superalgebra.   
\end{exa}

\begin{df}
A \textbf{noncommutative Hom-anti-pre-Poisson superalgebra} is a quintuple $(\mathcal A,\circ,\prec,\succ,\alpha)$ where $(\mathcal A,\circ,\alpha)$ is a Hom-anti-pre-Lie superalgebra and $(\mathcal A,\prec,\succ,\alpha)$ is a Hom-anti-dendriform superalgebra such that, the following conditions hold:
\begin{align}
 (x\circ y-(-1)^{|x||y|}y\circ x)\prec \alpha(z)&= -\alpha(x)\circ(y\prec z)+(-1)^{|x||y|}\alpha(y)\prec(x\circ z),\label{cond-ant-Hom-pre-poiss1}\\
 \alpha(x)\succ(y\circ z-(-1)^{|y||z|}z\circ y)&= -(-1)^{|x||y|}\alpha(y)\circ(x\succ z)+(-1)^{|x||y|}(y\circ x)\succ \alpha(z),\label{cond-ant-Hom-pre-poiss2}\\
 (x\prec y+x\succ y)\circ\alpha(z)&=-\alpha(x)\prec(y\circ z)-(-1)^{|y||z|}(x\circ z)\succ \alpha(y)\label{cond-ant-Hom-pre-poiss3}\\
 \alpha(x)\circ(y\prec z+y\succ z)&-(-1)^{|x|(|y|+|z|)}(y\prec z+y\succ z)\circ\alpha(x)-(-1)^{|x||y|}\alpha(y)\prec(x\circ z)\label{cond-ant-Hom-pre-poiss4}\\
 &-(x\circ y)\succ\alpha(z)=0.\nonumber
\end{align}
for all $x,y,z\in\mathcal{H}(\mathcal A)$.\\
A noncommutative Hom-anti-pre-Poisson superalgebra $(\mathcal A,\circ,\prec,\succ,\alpha)$ is called \textbf{multiplicative} if $(\mathcal A,\circ,\alpha)$ and $(\mathcal A,\prec,\succ,\alpha)$ are multiplicatives, and called \textbf{regular} if $\alpha$ is bijective.
\end{df}

The compatibility conditions defining noncommutative Hom-anti-pre-Poisson superalgebras can be viewed as a natural extension of Hom-type constructions using twisted $\mathcal{O}$-operators and Rota-Baxter operators (\cite{Mabrouk-Ncib-Silvestrov,Mabrouk-Silvestrov-Zouaidi}), ensuring a coherent interaction between the underlying Hom-anti-pre-Lie and Hom-anti-dendriform structures.
\begin{rem}
We recover noncommutative anti-pre-Poisson superalgebras structures when $\alpha=id$.     
\end{rem}


\begin{thm}\label{thm-twist-ant-pre-Poiss}
 Let $\alpha$ be an algebra homomorphism of a noncommutative anti-pre-Poisson superalgebra $(\mathcal A,\circ,\prec,\succ)$. Then, $(\mathcal A,\circ_\alpha,\prec_\alpha,\succ_\alpha,\alpha)$ is a multiplicative noncommutative Hom-anti-pre-Poisson superalgebra where $\prec_\alpha,\;\succ_\alpha$ and $\circ_\alpha$ are defined respectively by Eqs. \eqref{twist-anti-dendriform1}, \eqref{twist-anti-dendriform2} and \eqref{twist-anti-pre-Lie}.
\end{thm}
\begin{proof}
By Propositions \ref{twist-dendr} and \ref{prop-twist-anti-pre-Lie} we can see easily that $(\mathcal{A},\prec_\alpha,\succ_\alpha,\alpha)$ is a multiplicative Hom-anti-dendriform superalgebra and  $(\mathcal{A},\circ_\alpha,\alpha)$ is a multiplicative Hom-anti-pre-Lie superalgebra.\\
Let $x,y,z\in\mathcal H(\mathcal A)$. Then, by the fact that $(\mathcal A,\circ,\prec,\succ)$ is an anti-pre-Poisson superalgebra we have:
\begin{align*}
(x\circ_\alpha y-(-1)^{|x||y|}y\circ_\alpha x)\prec_\alpha \alpha(z)&=\alpha(\alpha(x)\circ \alpha(y)-(-1)^{|x||y|}\alpha(y)\circ \alpha(x))\prec \alpha^2(z)\\&=\alpha^2\big((x\circ y-(-1)^{|x||y|}y\circ x)\prec z)\big)\\&= \alpha^2\big(-x\circ(y\prec z)+(-1)^{|x||y|}y\prec(x\circ z)\big)\\&=-\alpha(x)\circ(y\prec z)+(-1)^{|x||y|}\alpha(y)\prec(x\circ z),   
\end{align*}
which implies that, condition \eqref{cond-ant-Hom-pre-poiss1} is satisfied. By the same way, we can show that, conditions \eqref{cond-ant-Hom-pre-poiss2}-\eqref{cond-ant-Hom-pre-poiss4} are satisfied. Then, $(\mathcal A,\circ_\alpha,\prec_\alpha,\succ_\alpha,\alpha)$ is a multiplicative Hom-anti-pre-Poisson superalgebra.
\end{proof}

\begin{prop}
Let $(\mathcal A, [\cdot,\cdot],\mu,\alpha)$ be a noncommutative Hom-Poisson superalgebra, $(\mathcal A,\circ,\alpha)$ be a compatible Hom-anti-pre-Lie superalgebra of $(\mathcal A,[\cdot,\cdot],\alpha)$ and $(\mathcal A,\prec,\succ,\alpha)$ a compatible Hom-anti-dendriform superalgebra of $(\mathcal A,\mu,\alpha)$. If $(\mathcal A,-\mathfrak L_\circ,-\mathfrak L_\prec,-\mathfrak R_\succ,\alpha)$ be a representation of $(\mathcal A,[\cdot,\cdot],\mu,\alpha)$, then $(\mathcal A,\prec,\succ,\circ,\alpha)$ is a Hom-anti-pre-Poisson superalgebra. Conversely, let $(\mathcal A,\circ,\prec,\succ,\alpha)$ be a Hom-anti-pre-Poisson superalgebra, $(\mathcal A,[\cdot,\cdot],\alpha)$ be the sub-adjacent Hom-Lie algebra of $(\mathcal A,\circ,\alpha)$ and $(\mathcal A,\mu,\alpha)$ be the sub-adjacent Hom-associative superalgebra of $(\mathcal A,\prec,\succ,\alpha)$. Then $(\mathcal A,[\cdot,\cdot],\mu,\alpha)$ is a Hom-Poisson superalgebra with a representation $(\mathcal A,-\mathfrak L_\circ,-\mathfrak L_\prec,-\mathfrak R_\succ,\alpha)$. In this case, we say $(\mathcal A,[\cdot,\cdot],\mu,\alpha)$ is the \textbf{sub-adjacent Hom-Poisson superalgebra} of $(\mathcal A,\prec,\succ,\circ,\alpha)$ and $(\mathcal A,\prec,\succ,\circ,\alpha)$ is a \textbf{compatible Hom-anti-pre-Poisson superalgebra} of $(\mathcal A,[\cdot,\cdot],\mu,\alpha)$.    
\end{prop}
\begin{proof} Let $x,y,z\in\mathcal{H}(\mathcal{A})$. By the fact that $(\mathcal A,-\mathfrak L_\circ,-\mathfrak L_\prec,-\mathfrak R_\succ,\alpha)$ is a representation of $(\mathcal A,[\cdot,\cdot],\mu,\alpha)$, we have:
\begin{align*}
&(x\circ y-(-1)^{|x||y|}y\circ x)\prec \alpha(z) +\alpha(x)\circ(y\prec z)-(-1)^{|x||y|}\alpha(y)\prec(x\circ z)\\=& 
\mathfrak L_\prec([x,y])\alpha(z)+\mathfrak L_\circ(\alpha(x))\mathfrak L_\prec(y)z-(-1)^{|x||y|}\mathfrak L_\prec(\alpha(y))\mathfrak L_\circ(x)z\\=&0.
\end{align*}

Similarly we have 
\begin{align*}
&\alpha(x)\succ(y\circ z-(-1)^{|y||z|}z\circ y)+\alpha(y)\circ(x\succ z)-(-1)^{|x||y|}(y\circ x)\succ \alpha(z)\\=&
\mathfrak R_\succ([y,z])\alpha(x)+\mathfrak L_\circ(\alpha(y))\mathfrak R_\succ(z)x-(-1)^{|x||y|}\mathfrak R_\succ(\alpha(z))\mathfrak L_\circ(y)x\\=&0,
\end{align*}
and
\begin{align*}
&(x\prec y+x\succ y)\circ\alpha(z)+\alpha(x)\prec(y\circ z)+(-1)^{|y||z|}(x\circ z)\succ \alpha(y)\\=&
\mathfrak L_\circ(\mu(x,y))\alpha(z)+\mathfrak L_\prec(\alpha(x))\mathfrak L_\circ(y)z+(-1)^{|y||z|}\mathfrak R_\succ(\alpha(y))\mathfrak L_\circ(x)z\\=&0.
\end{align*}
Then conditions \eqref{cond-ant-Hom-pre-poiss1}-\eqref{cond-ant-Hom-pre-poiss3} are satisfied.
\begin{align*}
0=&[\alpha(x),\mu(y,z)]+(-1)^{|x||y|}\mu([y,x],\alpha(z))+(-1)^{|x|(|y|
+|z|)}\mu(\alpha(y),[z,x])\\
=&\alpha(x)\circ(y\succ z+ y\prec z)-(-1)^{|x|(|y|+|z|)}(y\succ z-y\prec z)\circ \alpha(x)\\
&+(-1)^{|x||y|}(y\circ x-(-1)^{|x||y|}x\circ y)\succ \alpha(z)
+(-1)^{|x||y|}(y\circ x-(-1)^{|x||y|}x\circ y)\prec \alpha(z)\\
&+(-1)^{|x|(|y|+|z|)}\alpha(y)\succ(z\circ x-(-1)^{|x||z|}x\circ z)+(-1)^{|x|(|y|+|z|)}\alpha(y)\prec (z\circ x-(-1)^{|x||z|}x\circ z)\\
\overset{\eqref{cond-ant-Hom-pre-poiss1}-\eqref{cond-ant-Hom-pre-poiss3}}{=}&2\Big(\alpha(x)\circ(y\succ z+y\prec z)-(-1)^{|x|(|y|+|z|)}(y\succ z+y\prec z)\circ \alpha(x)-(-1)^{|x||y|}\alpha(y)\succ(x\circ z)\\&-(x\circ y)\prec \alpha(z)\Big),
\end{align*}
which gives condition \eqref{cond-ant-Hom-pre-poiss4}. The converse part is proved similarly.
\end{proof}
\begin{df}
 Let $(V,\rho,\mathfrak l,\mathfrak r,\beta)$ be a representation of a noncommutative Hom-Poisson superalgebra $(\mathcal A,[\cdot,\cdot],\mu,\alpha)$. An even linear map $T:V\to\mathcal A$ is called an \textbf{anti-super-$\mathcal O$-operator} on $(\mathcal A,[\cdot,\cdot],\mu,\alpha)$ if it satisfying both \eqref{cond-ant-O-oper-ncomm-Hom-ass2} and \eqref{cond-O-Oper-Hom-Lie2}.
 
 an anti-super-$\mathcal O$-operator $T$ on a Hom-Poisson superalgebra $(\mathcal A,[\cdot,\cdot],\mu,\alpha)$ is called \textbf{Strong} if:
 \begin{enumerate}
\item $T$ is strong on the Hom-Lie superalgebra $(\mathcal A,[\cdot,\cdot],\alpha)$.
\item $T$ is strong on the Hom-associative superalgebra $(\mathcal A,\mu,\alpha)$.
\item $T$ satisfying for any homogeneous elements $u,v,w$ of $V$ the following condition:
\begin{equation}\label{cond-strong-O-op-Hom-Pois-sup}
   \mathfrak l([T(u),T(v)])\beta(w)+(-1)^{|v||w|}\mathfrak r([T(u),T(w)])\beta(v)+(-1)^{|u|(|v|+|w|)}\rho(\mu(T(v),T(w)))\beta(u)=0. 
\end{equation}
 \end{enumerate}
 In particular, an anti-super-$\mathcal O$-operator $\mathcal R$ of $(\mathcal A,[\cdot,\cdot],\mu,\alpha)$ associated to the adjoint representation $(\mathcal A,\mathfrak {ad},\mathfrak L,\mathfrak R,\alpha)$ is called an \textbf{anti-Rota-Baxter operator (of weight zero)}, that is, $\mathcal R:\mathcal A\to\mathcal A$ is an even linear map satisfying conditions \eqref{ant-RB-oper-noncomm-ass} and \eqref{cond-Rota-Baxter-Oper}.
An anti-Rota-Baxter operator $\mathcal R$ is called \textbf{strong}, if it satisfies \eqref{strong-ant-RB-oper-noncomm-ass}, \eqref{cond-strong-Rota-Baxter-Oper} and the following condition.
\begin{equation}\label{cond-strong-R-B-Op-Hom-Poiss}
[[\mathcal R(x),\mathcal R(y)],\alpha(z)]+(-1)^{|y||z|}[[\mathcal R(x),\mathcal R(z)],\alpha(y)]+(-1)^{|x|(|y|+|w
z|)}[\mu(\mathcal R(y),\mathcal R(z)),\alpha(x)]=0,   
\end{equation}
for all $x,y,z\in\mathcal H(\mathcal A)$.
\end{df}
\begin{thm}
Let $T:V\rightarrow \mathcal{A}$ be an anti-super-$\mathcal{O}$-operator on a noncommutative Hom-Poisson-superalgebra $(\mathcal A,[\cdot,\cdot],\mu,\alpha)$ 
with respect to a representation $(V,\rho,\mathfrak l,\mathfrak r,\beta)$. Define the following binary operations $\circ_T,\prec_T,\succ_T:V\times V\rightarrow V$ defined by
\begin{align*}
u\circ_T v&=-\rho(T(u))v \\
u\prec_T v&=-\mathfrak l(T(u))v \\
u\succ_T v&=-(-1)^{|u||v|}\mathfrak r (T(v))u
\end{align*}
Then $(V,\circ_T,\prec_T,\succ_T,\beta)$ is a Hom-anti-Pre-Poisson superalgebra if and only if $T$ is strong.
\end{thm}
\begin{proof}
 Theorem \ref{Hom-ant-pre-Lie-by-ant-O-oper} and Proposition \ref{Hom-ant-dend-by-ant-O-oper} resectively gives that, $(V,\circ_T,\beta)$ is a Hom-anti-pre Lie superalgebra and $(V,\succ_{T},\prec_{T},\beta)$ is a Hom-anti-dendriform superalgerbra if and only if $T$ is strong.\\
 For any homogeneous elements $u,v,w$ of $V$, we have:
\begin{align*}
 &(u\circ_{T} v-(-1)^{|u||v|} v\circ_{T} u)\prec_{T} \beta(w)+\beta(u)\circ_{T}(v\prec_{T}w)-(-1)^{|u||v|}\beta(v)\prec_{T} (u\circ_{T}w)\\
 =&-\mathfrak l(T(-\rho(T(u))v+(-1)^{|u||v|}\rho(T(v))u))\beta(w)+\rho(T(\beta(u)) l(T(v))w-(-1)^{|u||v|}\mathfrak l(T(\beta(v)))\rho(T(u))w \\
 =&-\mathfrak l([T(u),T(v)])\beta(w)+\rho(\alpha (T(u)))\mathfrak  l(T(v))w-(-1)^{|u||v|}\mathfrak l(\alpha (T(v)))\rho(T(u))w\\
 =&0.
\end{align*}
The last equality found using Eq.\eqref{cond-rep-noncom-Hom-Pois1}.\\
Similarly, by using Eqs.\eqref{cond-rep-noncom-Hom-Pois2} and \eqref{cond-rep-noncom-Hom-Pois3}, we have: 
\begin{align*}
&\beta(u)\succ_{T}(v\circ_{T} w-(-1)^{|v||w|}w\circ_{T} v)+\beta(v)\circ_{T}(u\succ_{T} w)-(-1)^{|u||v|}(v\circ u)\succ_{T} \beta(w)\\&=-
\mathfrak r([T(v),T(w)])\beta(u)+\rho (\alpha(T(v)))\mathfrak r_{T}(T(w))u-(-1)^{|u||v|}\mathfrak r(\alpha(T(w)))\rho(T(v))u\\
&=0,
\end{align*}
and 

\begin{align*}
&(u\prec_{T}v+u\succ_{T}v)
\circ_{T} \beta(w)+\beta(u)\prec_{T}(v\circ_{T}w)+(-1)^{|v||w|}(u\circ_{T}w)\succ_{T} \beta(v)\\
=&-\rho(\mu(T(u),T(v))\beta(w)+\mathfrak l(\alpha(T(u)))\rho(T(v))w+(-1)^{|v||w|} \mathfrak r(\alpha(T(v))\rho(T(u))w \\
=&0.
\end{align*}

It remains to be shown that condition \eqref{cond-ant-Hom-pre-poiss4} is satisfied. Fo all $u,v,w\in\mathcal H(V)$, we have:
\begin{align*}
LHS=&\beta(u)\circ_T(v\prec_T w+v\succ_Tw)-(-1)^{|u|(|v|+|w|)}(v\prec_T w+v\succ_Tw)\circ_T\beta(u)-(-1)^{|u||v|}\beta(v)\prec_T(u\circ_Tw)\\&-(u\circ_Tv)\succ_Tw\\=&\rho(\alpha(T(u)))\mathfrak l(T(v))w+(-1)^{|v||w|}\rho(\alpha(T(u)))\mathfrak r(T(w))v+(-1)^{|u|(|v|+|w|)}\rho(\mu(T(v),T(w)))\beta(u)\\&-(-1)^{|u||v|}\mathfrak l(\alpha(T(v)))\rho(T(u))w-(-1)^{|w|(|u|+|v|)}\mathfrak r(\alpha(T(w)))\rho(T(u))v\\=&\Big(\rho(\alpha(T(u)))\mathfrak l(T(v))-(-1)^{|u||v|}\mathfrak l(\alpha(T(v)))\rho(T(u))\Big)w\\&+\Big((-1)^{|v||w|}\rho(\alpha(T(u)))\mathfrak r(T(w))-(-1)^{|w|(|u|+|v|)}\mathfrak r(\alpha(T(w)))\rho(T(u))\Big)v \\&+ (-1)^{|u|(|v|+|w|)}\rho(\mu(T(v),T(w)))\beta(u) \\=&\mathfrak l([T(u),T(v)])\beta(w)+(-1)^{|v||w|}\mathfrak r([T(u),T(w)])\beta(v)+(-1)^{|u|(|v|+|w|)}\rho(\mu(T(v),T(w)))\beta(u). 
\end{align*}
Then $LHS=0$ if and only if condition \eqref{cond-strong-O-op-Hom-Pois-sup} is satisfied which is, $T$ is strong. The proof is finished.
\end{proof}
\begin{cor}
Let $\mathcal R:\mathcal A\rightarrow \mathcal{A}$ be an anti-Rota-Baxter operator of weight zero on a noncommutative Hom-Poisson-superalgebra $(\mathcal A,[\cdot,\cdot],\mu,\alpha)$. Define the following binary operations $\circ_\mathcal{R},\prec_\mathcal{R},\succ_\mathcal{R}:\mathcal{A}\times\mathcal{A}\rightarrow \mathcal{A}$ defined for any $x,y\in\mathcal{H}(\mathcal A)$ by
\begin{align*}
x\circ_\mathcal{R} y&=-[\mathcal R(x),y] \\
x\prec_\mathcal{R} y&=-\mu(\mathcal R(x),y) \\
x\succ_\mathcal{R} y&=-\mu(x,\mathcal R(y)).
\end{align*}
Then $(\mathcal A,\circ_\mathcal{R},\prec_\mathcal{R},\succ_\mathcal{R},\alpha)$ is a Hom-anti-pre-Poisson superalgebra if and only if $\mathcal R$ is strong.    
\end{cor}

\end{document}